\documentclass[11pt]{amsart}
\usepackage{amsmath,amssymb,amsfonts}
\usepackage[all]{xy}

\usepackage[T5,T1]{fontenc}
\usepackage{mathpazo}
\usepackage{hyperref}
\usepackage{a4wide}

\theoremstyle{plain}
\newtheorem{thm}{Theorem}[section]
\newtheorem{lem}[thm]{Lemma}
\newtheorem{cor}[thm]{Corollary}
\newtheorem{prop}[thm]{Proposition}

\newtheorem{conj}[thm]{Conjecture}

\theoremstyle{definition}
\newtheorem{defn}[thm]{Definition}
\newtheorem{ex}[thm]{Example}

\newtheorem{question}[thm]{Question}

\newtheorem{rmk}[thm]{Remark}
\newtheorem{rmks}[thm]{Remarks}

\newcommand{\cd}{{\rm cd}}

\newcommand{\im}{{\rm im}}
\newcommand{\Spec}{{\rm Spec \,}}

\newcommand{\tr}{{\rm tr}}
\newcommand{\trg}{{\rm trg}}

\newcommand{\Gal}{{\rm Gal}}

\newcommand{\sC}{{\mathcal C}}

\newcommand{\sE}{{\mathcal E}}

\newcommand{\sR}{{\mathcal R}}

\newcommand{\A}{{\mathbb A}}

\newcommand{\C}{{\mathbb C}}

\newcommand{\F}{{\mathbb F}}
\newcommand{\G}{{\mathbb G}}

\renewcommand{\P}{{\mathbb P}}
\newcommand{\Q}{{\mathbb Q}}

\newcommand{\U}{{\mathbb U}}

\newcommand{\Z}{{\mathbb Z}}

\def\NDT{{\fontencoding{T5}\selectfont Nguy\~ \ecircumflex n Duy T\^an}}
\def\NQT{{\fontencoding{T5}\selectfont Nguy\~ \ecircumflex n Qu\'\ocircumflex{}c Th\'\abreve{}ng}} 
\begin{document}
\title{Triple Massey Products and Galois theory} 

\dedicatory{Dedicated to H\'el\`ene Esnault}
 \author{ J\'an Min\'a\v{c} and \NDT}
\address{Department of Mathematics, Western University, London, Ontario, Canada N6A 5B7}
\email{minac@uwo.ca}
 \address{Department of Mathematics, Western University, London, Ontario, Canada N6A 5B7 and Institute of Mathematics, Vietnam Academy of Science and Technology, 18 Hoang Quoc Viet, 10307, Hanoi - Vietnam } 
\email{dnguy25@uwo.ca}
\thanks{%2010 {\it Mathematics Subject Classification}: Primary 12 F12; Secondary 55 S30.\\
The first author is supported in part by the Natural Sciences and Engineering Research Council of Canada (NSERC) grant R0370A01. The second author is supported in part by the National Foundation for Science and Technology Development (NAFOSTED)}

 \begin{abstract}
%Let $F$ be any field.
 We show that any triple Massey product with respect to prime 2 contains 0 whenever it is defined over any field. This extends the theorem of M. J. Hopkins and K. G. Wickelgren, from global fields to any fields. This is the first time when the vanishing of any $n$-Massey product for some prime $p$ has been established for all fields. This leads to a strong restriction on the shape of relations in the maximal pro-2-quotients of absolute Galois groups, which was out of reach until now. We also develop an extension of  Serre's 
transgression method to detect triple commutators in relations of pro-$p$-groups, where we do not require that all cup products vanish.
We prove that all $n$-Massey products, $n\geq 3$, vanish for  general Demushkin groups. We formulate and provide evidence for two conjectures related to the structure of absolute Galois groups of fields. In each case when these conjectures can be verified, they have some interesting concrete Galois theoretic consequences. They are also related to the Bloch-Kato conjecture.

\end{abstract}
\maketitle
\section{Introduction}
A major problem in Galois theory is the characterization of profinite groups which are realizable as absolute Galois groups of fields. This is a difficult problem, and in general little is known. In our paper we provide a definite contribution valid for all fields. 
 
In 1967 A. Weil in \cite{Wei} describing Artin's first result in the theory of real fields says "Even now, this is an altogether isolated result of great depth, whose significance for the future is not to be assessed lightly." In the classical papers \cite{AS1,AS2} published in 1927, E. Artin and O. Schreier went on with developing a theory of real fields and showed in particular that the only non-trivial finite subgroups of absolute Galois groups are cyclic groups of order 2. In \cite{Be}, E. Becker developed some parts of Artin-Schreier theory by replacing separable closures of fields by maximal $p$-extensions of fields. Here and below, by $p$ we mean a prime number. %In \cite[Theorem 11.6.2]{FJ} it was shown that a profinite group is projective if and only if it is the absolute Galois group of a PAC field. 
The notions of projective profinite fields and pseudo algebraically closed (PAC) fields are now basic notions in Galois theory. (See \cite[Chapter 11]{FJ} and also the definition after Conjecture~\ref{conj:vanishing n-Massey} and the beginning of Section 4 below.) In \cite{Ax} it was proved that if a field $F$ is PAC  then its absolute Galois group is projective. However the actual name PAC was introduced only later by G. Frey in \cite{Frey} following a suggestion of M. Jarden. A nice proof due to D. Haran is presented in \cite[Theorem 11.6.2]{FJ}. In fact in \cite{LvdD} A. Lubotzky and L. van den Dries proved that for any given projective group $G$ there exists a PAC field such that its absolute Galois group is isomorphic to $G$. See also \cite[Corollary 23.1.2]{FJ}. 
Further, Y. L. Ershov in \cite{Er} showed that if finitely many profinite groups are absolute Galois groups, so is their free product. (See also I. Efrat and D. Haran's related result in \cite{EH} concerning  pro-$p$-groups as absolute Galois groups.)
%These two latter results 
These results above concerning projective groups and free products were generalized in a far-reaching way to ''profinite groups that are relatively projective with respect to appropriate subsets of closed subgroups'' in \cite{Koe2} and \cite{HJP}. 
In the remarkable paper \cite{Koe1} in 2001, J. Koenigsmann provided a classification of solvable absolute Galois groups. In \cite{MS1} it was shown that orderings of fields can be detected already by much smaller Galois $2$-extensions than maximal 2-extensions. In 1996,  \cite{MS2}, using Villegas' results in \cite{Vi} provided a structural result of the quotient of absolute Galois groups by the third 2-descending series. These results were extended to analogous results for $p$-descending series and $p$-Zassenhaus series in \cite{EM1,EM2}. These are a few fundamental results on the structure of absolute Galois groups of general fields.

However in a recent spectacular development, the Bloch-Kato conjecture was proved by M. Rost and V. Voevodsky. (See \cite{Voe}.) These are very strong restrictions on the structure of absolute Galois groups but these results do not give directly structural results of absolute Galois groups. However in \cite{MS2, EM1, EM2}, the previous results by A. Merkurjev and A. Suslin \cite{MSu}) on the Bloch-Kato conjecture in degree 2 were used. 
It is a challenging important problem both for the structure of absolute Galois groups as well as for understanding the Bloch-Kato conjecture better, to provide a direct precise translation of the Bloch-Kato conjecture on the group-theoretical properties of absolute Galois groups. 
Building on the work of a number of mathematicians (\cite{Dwy75, DGMS, Ef, GLMS, EM1, EM2, HW, MS2, MSu,Vi,Voe}) we formulate here two other fundamental and strong conjectures which we call the "Vanishing $n$-Massey Conjecture" and the "Kernel $n$-Unipotent Conjecture".

The main objective of this paper is to prove the Vanishing $3$-Massey Conjecture for prime $2$ for all fields and to derive strong consequences for the structure of relations in absolute Galois groups of all fields or their maximal pro-$2$ quotients. 
Let us first recall briefly the notion of triple Massey products (see Section 2 for  more detail on Massey products). Let $\sC^\bullet$ be a differential graded algebra with differential $\partial: C^\bullet \to C^{\bullet+1}$ and homology $H^\bullet$. Suppose that $a,b,c\in H^1$ such that $ab=bc=0$.  We can choose $A,B,C$ in $\sC^1$ representing $a,b,c$ respectively. Since $ab=0$, there is $E_{ab}$ such that $\partial E_{ab}=AB$, similarly there is $E_{bc}$ such that $\partial E_{bc}=BC$.  Note that $\partial(E_{ab}C+A E_{bc})=0$, hence $E_{ab}C+A E_{bc}$ represents an element of $H^2$. The set of all $E_{ab}C+A E_{bc}$ obtained in this manner is defined to be the triple Massey product $\langle a,b,c\rangle \subset H^2$. We say  that the triple Massey product vanishes if it contains 0. 

Now let $F$ be a field of characteristic $\not=2$ and let $G=G_F(2)$ be the maximal pro-2  quotient of the absolute  Galois group $G_F$ of $F$. Let $\sC^\bullet =(\sC^\bullet(G,\F_2),\partial)$ denote the differential graded algebra of $\F_2$-inhomogeneous cochains in the continuous group cohomology of $G$ (see the first paragraph in Section~\ref{sec:3} for more details).  For any $a\in F^*=F\setminus \{0\}$, let $\chi_a$ denote the corresponding character  via the Kummer map $F^*\to H^1(G,\F_2)$. 
In the work of M. J. Hopkins and K. G. Wickelgren \cite{HW}, the following result was proved.
\begin{thm}[{\cite[Theorem 1.2]{HW}}]
\label{thm:HW}
 Let  $F$ be a global field of characteristic $\not=2$ and $a,b,c\in F^*$. The triple Massey product $\langle \chi_a,\chi_b,\chi_c\rangle$ contains 0 whenever it is defined.
\end{thm} 
 
In our paper we show that triple Massey products with respect to prime 2 vanish over any field $F$. As it follows from Example~\ref{ex:free} and from Witt's Theorem (see \cite{Wi}, \cite[Theorem 9.1]{Ko2}) that  $n$-fold Massey products vanish with respect to $2$ if ${\rm char}(F) =2$. So we can assume that the characteristic of $F$ is not 2.
%Namely, we have
\begin{thm}
\label{thm:vanishing}
 Let $F$ be an arbitrary field of characteristic $\not=2$, $a,b,c\in F^*$. The triple Massey product $\langle \chi_a,\chi_b,\chi_c\rangle$ contains 0 whenever it is defined.
\end{thm}

This has remarkable consequences for the structures of absolute Galois groups $G_F$ and their maximal pro-2-quotients $G_F(2)$. We state our results for finitely generated pro-2-groups but our methods can be used also in the case of infinitely generated pro-2-groups with several relations. In Section 7 we also consider pro-$p$-groups for $p$ possibly not equal to 2. The reason for our restriction in the remainder of the paper for considering $p=2$ is that we do not yet have complete results for triple Massey products for $p>2$. This is work in progress. (See \cite{GMT}.) 
The results on the shape of relations of finitely generated pro-2-groups of the form $G_F(2)$ for some field $F$ are  fundamental results extending the classical results of S. P. Demushkin, K. Iwasawa, U. Jannsen, H. Koch, J. Labute, J.-P. Serre, I. Shafarevich and K. Wingberg. (See e.g., \cite{De1, De2, I,  JaWi, Ko1, Ko2, La2, Se, Sh}.)
Thus we provide  strong restrictions on the structure of  groups $G_F(2)$.
Before stating the results we illustrate them with an example. Examining the classification of Demushkin group by Labute in \cite{La2} one sees that $G_F(p)$ always has a presentation where the  generating relation is a product of commutators between generators and $p$-powers of generators. (If $G_F(p)$ for a local field is not a Demushkin group, then it is free pro-$p$.) Already in the paper \cite{CEM}, Section 9, it  was shown that $G:=S/\langle [[x_1,x_2],x_3]\rangle$, where $S$ is a free pro-2-group on generators $x_1,\ldots,x_n$, $n\geq 3$, cannot be an absolute Galois group of any field. (In this paper, for $r\in S$, we denote  by $\langle r\rangle$ the closed normal subgroup of $S$ generated by $r$.) 
One can also deduce, for example, that $G$ as above cannot be isomorphic  to $G_F(2)$ for any field $F$. However relations where simple commutators are combined with triple ones like $r=[x_4,x_5][[x_2,x_3],x_1]$ are much harder to exclude, and one could not show that 
$G=S/\langle r \rangle$, $S$ is a free pro-2-group on $n$ generators $x_1,\ldots, x_n$ with $n\geq 5$, is not isomorphic to  $G_F(2)$ for any field $F$ until this work.  In Examples~\ref{first example} and \ref{ex:ob}, we deal with this group in a detailed way, and in particular we show that $G\not\simeq G_F(2)$ for  any field $F$. The next Theorem~\ref{thm:main ob1} and Theorem~\ref{thm:main ob2} are a vast generalization of this example.

That some conditions are necessary, one can see for example from the following example. Consider a free pro-2-group $S$ on generators $x_1,x_2,x_3$ and 
\[
 G=S/\langle r\rangle, \quad r=[x_1,x_3].
\]
 However now consider three new generators, $y_1=x_1x_2^{-1}$, $y_2=x_2$, $y_3=x_3$ of $S$. Then 
\[
 \begin{aligned}
  r&=[y_1y_2,y_3]=[y_1,y_3]\cdot [[y_1,y_3],y_2]\cdot [y_2,y_3]\\
&= [y_1,y_3]\cdot [y_2,y_3]\cdot [[y_1,y_3],y_2] \cdot r^\prime,
 \end{aligned}
\]
where $r^\prime$ is an element in  the 4-th term $S_{(4)}$ in the 2-Zassenhaus filtration of $S$ defined in Section 3 after the proof of Lemma~\ref{lem:Sharifi}. Observe now first that 
\[
 G_1=S_1/\langle [x_1,x_3]\rangle,
\]
where $S_1$ is a free pro-2-group on generators $x_1, x_3$, is realizable as $G_{F_1}(2)$, over the field $\C((X_1))((X_2))$ of iterated power series (see \cite[Corollary 3.9, part (2)]{Wa}). Also $G_2:=S_2$, the free pro-2 group on $x_2$ is realizable as $G_{F_2}(2)$ where $F_2=\C((X_2))$. By \cite[Theorem 3.6]{JW}, we see that their free product in the category of pro-$2$-groups
\[
 G=G_1\ast G_2,
\]
is also realizable as $G_F(2)$ for some field $F$. Hence
\[
 G=S/\langle [y_1,y_3]\cdot [y_2,y_3]\cdot [[y_1,y_3],y_2] \cdot r^\prime\rangle,
\]
where $S$ is a free pro-2-group on generators $y_1,y_2,y_3$, is of the form $G_F(2)$. Hence we see that some conditions as in our Theorems~\ref{thm:main ob1} and \ref{thm:main ob2} are necessary to guarantee the truth of these theorems. Therefore these conditions look like natural conditions. It is clear that they are very strong conditions and they extend some results on the shape of relations of $G_F(2)$ from local fields to all fields.

In the theorems below we use the following notation.
Let $(I,<)$ be a well-ordered set. 
Let $S$ be a free pro-2 group on  a set of generators $x_i, i\in I$ (see \cite[Definition 3.5.14]{NSW}). Let $S_{(i)}$, $i=1,2,\ldots$ be the 2-Zassenhaus filtration of $S$. (See Section 3 for a definition of Zassenhaus filtration.) 
Then any element $r$ in $S_{(2)}$ may be written uniquely as
\begin{equation}
r= \prod_{i\in I} x_i^{2a_i} \prod_{i<j}[x_i,x_j]^{b_{ij}}\prod_{i< j, k\leq j}[[x_i,x_j],x_k]^{c_{ijk}}r^\prime,  
\label{modulo S4}
\end{equation}
where $a_i,b_{ij},c_{ijk}\in \{0,1\}$ and $r^\prime\in S_{(4)}$. For convenience we call (\ref{modulo S4}) the canonical decomposition modulo $S_{(4)}$ of $r$ (with respect to the basis $(x_i)$) and we also set $u_{ij}=b_{ij}$ if $i<j$, and  $u_{ij}=b_{ji}$ if $j<i$.

\begin{thm}
\label{thm:main ob1}
Let $\sR$ be a set of elements in $S_{(2)}$. Assume that there exists an element $r$ in $\sR$ and distinct indices $i,j,k$ with $i<j, k<j$ such that: 
\begin{enumerate}
\item[(i)]  In (\ref{modulo S4})  the canonical decomposition modulo $S_{(4)}$ of $r$, 
$a_k=a_j=u_{ij}=u_{kj}=u_{ki}=u_{kl}=u_{jl}=0$  for all $l\neq i,j,k$, and  $c_{ijk}\not=0$; and
\item[(ii)]  for every $s\in \sR$ which is different from  $r$,  the factors $[x_k,x_i]$, $[x_i,x_k]$ and $[x_i,x_j]$ do not occur in the canonical decomposition modulo $S_{(4)}$ of $s$. %Here the indices $i,j,k$ are as in (i).
\end{enumerate}
Then $G=S/\langle \sR \rangle$ is not realizable as $G_F(2)$ for any field $F$.
\end{thm}

\begin{thm}
\label{thm:main ob2}
Let $\sR$ be a set of elements in $S_{(2)}$. Assume that there exists an element $r$ in $\sR$ and distinct indices $i<j$ such that: 
\begin{enumerate}
\item[(i)]  In (\ref{modulo S4}) the canonical decomposition modulo $S_{(4)}$ of $r$,  
$a_i=a_j=u_{ij}=u_{il}=u_{jl}=0$, for all $l\not=i,j$ and $c_{iji}\not=0$ (respectively, $c_{ijj}\not=0$); and
\item[(ii)] for every $s\in \sR$ which is different from  $r$,  the factors $[x_i,x_j]$ and $x_i^2$ (respectively, $[x_i,x_j]$ and $x_j^2$) do not occur in the canonical decomposition modulo $S_{(4)}$ of $s$. %Here the indices $i,j$ are as in (i).
\end{enumerate}
Then $G=S/\langle \sR \rangle$ is not realizable as $G_F(2)$ for any field $F$.
\end{thm}

Theorem~\ref{thm:main ob1} (respectively Theorem~\ref{thm:main ob2}) follows immediately from Theorem~\ref{thm:vanishing} and Theorem~\ref{thm:ob1} (respectively Theorem~\ref{thm:ob2}).  
\begin{rmks}
1) Notice that any pro-$2$-group which is realizable as $G_F$  for some field $F$, is also realizable as $G_F(2)$. Hence the above two theorems also provide pro-2-groups which cannot be realizable as  the absolute Galois group of any field $F$.

2) One can also use Theorems~\ref{thm:main ob1} and \ref{thm:main ob2} to obtain {\it profinite} groups which are not realizable as the absolute Galois group of any field $F$. For simplicity we consider only the following example. Let $S$ be a free profinite group on $5$ generators $x_1, \ldots,x_5$ and let $r=[x_4,x_5][[x_2,x_3],x_1]$. Then $G:=S/\langle r \rangle$ cannot be realizable as $G_F$ for any field  $F$. In fact, one can check that the pro-$2$-quotient $G(2)$ of $G$ has a presentation $G(2)=S^\prime/\langle r' \rangle$, where $S'$ is a free pro-2-group on $5$ generators $y_1,\ldots,y_5$ and $r^\prime=[y_4,y_5][[y_2,y_3],y_1]$. By Theorem~\ref{thm:main ob1}, we see that $G(2)$ cannot be of the form $G_F(2)$ for any field $F$. Therefore $G$ is not realizable as $G_F$.
\end{rmks}

Motivated by the theorems above, we formulate the Vanishing $n$-Massey Conjecture for $n\geq 3$. See Definition~\ref{defn:vanishing n-Massey} for the definition of the vanishing $n$-fold Massey product property.
\begin{conj}
\label{conj:vanishing n-Massey}
 Let $p$ be a prime number and $n\geq 3$ an integer. Let $F$ be a field, which contains a primitive $p$-th root of unity if ${\rm char}(F)\not=p$. Then the absolute Galois group $G_F$ of $F$ has the vanishing $n$-fold Massey product property with respect to $\F_p$.
\end{conj}

A family of fields which satisfy the Vanishing $n$-Massey Conjecture for any $n\geq 3$ (and any $p$) are PAC fields. (Recall that a field $F$ is called PAC if each non-empty variety defined over $F$ has an $F$-rational point (see \cite[Chapter 11, page 192]{FJ}).)
This follows from the result mentioned earlier in the Introduction, that the absolute Galois groups of PAC fields are projective and from Example~\ref{ex:projective}. 

In this paper, Theorem~\ref{thm:vanishing}, more precisely Theorem~\ref{thm:Massey vanishing}, shows that Conjecture~\ref{conj:vanishing n-Massey} holds true for $n=3$, $p=2$ and for any field $F$. In \cite{MT2}, we show that the conjecture is true for any $n\geq 3$, $p>2$ and for any $p$-rigid field $F$. In \cite{MT3}, the conjecture is verified for $n=3$, $p>2$ and $F$ an algebraic number field. Note also that Theorem~\ref{thm:Demushkin} shows that the conjecture is true for any $n\geq 3$, any prime number $p$ and any local field $F$. 
Further results related to Conjecture~\ref{conj:vanishing n-Massey} are Proposition~\ref{prop:free product} and Proposition~\ref{prop:p-Massey} as well as additional results in \cite{MT2}.
In Section~\ref{sec:further}, we also formulate a related conjecture, the Kernel $n$-Unipotent Conjecture (see Conjecture~\ref{conj:kernel intersection}). 

As will be explained in Section~\ref{sec:further}, the Kernel $n$-Unipotent Conjecture evolved over a number of years through work contained in \cite{Vi},\cite{MS2},\cite{GLMS},\cite{EM1},\cite{EM2} and \cite{Ef}. This conjecture has  significant value because it describes specific pro-$p$-groups which are images of unipotent representations of absolute Galois groups as building blocks of quotients of absolute Galois groups by various terms in their $p$-Zassenhaus filtrations.

The Vanishing $n$-Massey Conjecture can be used to construct these building blocks from much smaller $p$-groups inductively. (See Theorem~\ref{thm:Dwyer}, due to B. Dwyer, and our use of it in Section 6.) Thus these two conjectures together provide us with valuable  tools for telling us which Galois $p$-groups we should be able to construct automatically from smaller Galois groups, and how we can proceed to build entire maximal $p$-extensions of any field. 
Our paper contributes to the developments of new directions in studies of Galois $p$-extensions of fields. It complements methods in current research in abelian birational geometry (\cite{BT}, \cite{BT2} and \cite{Pop}).

In retrospect we now understand the initial Artin-Schreier results from this new point of view, and we now better appreciate A. Weil's intuition about the significance of these results for future developments in Galois theory. (See Remark~\ref{rmk:Artin-Schreier}.)

It seems that our use of triple Massey products for detecting higher commutators is  the first time when the rather restrictive assumption that all cup products have to vanish, was removed. (See e.g., \cite{Ef,Ga,Mor,Vo,Vo2}.) In fact this suggests that there is a comprehensive extension of the theory described in \cite[Appendix]{Vo2} where the assumption on the relations of $G$ contained in a large enough weight of the free group mapping on $G$ can be considerably weakened if $G=G_F(p)$ for some prime $p$. (Here $G_F(p)$ is the maximal pro-$p$ quotient of the absolute Galois group $G_F$.) Work on this theory is in  progress. (See \cite{GMT}.)

In the following discussion, we refer to definitions for the formality of differential graded algebras and the motivation for studying formality to \cite{DGMS} as well as connections with Massey products. (To recall the notion of differential graded algebras abbreviated as DGAs, see  Section 2.) Let $\sC^\bullet:=\sC^\bullet(\Spec\, F,\Z/2) = \sC^\bullet(G_F,\Z/2)$ be the DGA of inhomogeneous continuous cochains of $G_F$ with coefficients in  $\Z/2$.
In the  paper \cite{HW}, the following extremely interesting question was posed.
\begin{question}[{\cite[Question 1.3]{HW}}] 
Is $\sC^\bullet({\rm Spec }\, F, \Z/2)$ formal?
\end{question}
It is known that if $\sC^\bullet(\Spec\, F,\Z/2)$ is formal, then all higher Massey products vanish. Therefore the vanishing property of Massey products makes the question above a natural one.

The structure of our paper is as follows. In Sections 2 and 3, basic facts on Massey products are reviewed. Some examples on groups satisfying the vanishing Massey product property are discussed in Section 4. In Section 5 we provide the first proof of Theorem~\ref{thm:vanishing} using splitting varieties \cite{HW}. In Section 6 we present the second proof of Theorem~\ref{thm:vanishing} using Galois theory and some results in \cite{GLMS}. 
In Section 7 we apply our results to show some strong restrictions on the shape of relations of $G_F(2)$, for a field $F$.
In the last section we point out certain notions related to our results and possibly interesting directions for further research.
\\
\\
{\bf Acknowledgements: } We are very grateful to Ido Efrat for his interest, encouragement, detailed notes, and suggested various improvements which we have used. We thank Michael Hopkins and Kirsten Wickelgren for  interesting discussions and correspondence, as well as their encouraging welcome of our results.   We thank  Murray Schacher and Suresh 
Venapally for their correspondence, and for information that they also obtained some of our results in Section 5.
We also thank Julien Blondeau, Sunil Chebolu, Brian Conrad, Jochen G\"{a}rtner, Detlev Hoffmann, John Labute, Jeffrey Lagarias, Christian Maire, Claudio Quadrelli, Andrew Schultz, Romyar Sharifi, \NQT, Adam Topaz, Zack Wolske and Lei Zhang  for their interest, encouragement and interesting discussions.

We are also grateful to the referees for their comments and valuable suggestions which we used to improve our exposition.

\section{Review of Massey products}
In this section and the next section, we review some basic facts about Massey products and we  use as main resources \cite{Dwy75}, \cite{Ef}, \cite{HW} and \cite{Wic1}. For other references on Massey products, see e.g., \cite{Fe, Kra,May, Mor, Vo}.

Let $A$ be a unital commutative ring. Recall that a differential graded algebra (DGA) over $A$ is a graded $A$-algebra 
\[\sC^\bullet =\oplus _{k\geq 0} \sC^k=\sC^0\oplus\sC^1\oplus \sC^2\oplus \cdots \] with product $\cup$ and equipped with a differential $\partial\colon \sC^\bullet\to \sC^{\bullet +1}$ such that 
\begin{enumerate}
\item $\partial$ is a derivation, i.e.,
\[
\partial(a\cup b) = \partial a\cup b +(-1)^k a\cup \partial b \quad (a\in \sC^k);
\]
\item $\partial^2=0$.
\end{enumerate}
Then as usual the cohomology is $H^\bullet:=\ker\partial/\im\partial$. We shall assume that $a_1,\ldots,a_n$ are elements in $H^1$.

\begin{defn}
 A collection $M=(a_{ij})$, $1\leq i<j\leq n+1$, $(i,j)\not=(1,n+1)$, of elements of $\sC^1$ is called a {\it defining system} for the {\it $n$th order Massey product} $\langle a_1,\ldots,a_n\rangle$ if the following conditions are fulfilled:
\begin{enumerate}
\item $a_{i,i+1}$ represents $a_i$.
\item $\partial a_{ij}= \sum_{l=i+1}^{j-1} a_{il}\cup a_{lj}$ for $i+1<j$.
\end{enumerate}
Then $\sum_{k=2}^{n} a_{1k}\cup a_{k,n+1}$ is a $2$-cocycle.
Its  cohomology class in $H^2$  is called the {\it value} of the product relative to the defining system $M$,
and is denoted by $\langle a_1,\ldots,a_n\rangle_M$.

The product $\langle a_1,\ldots, a_n\rangle$ itself is the subset of $H^2$ consisting of all  elements which can be written in the form $\langle a_1,\ldots, a_n\rangle_M$ for some defining system $M$. The product $\langle a_1,\ldots, a_n\rangle$ is {\it uniquely defined} if it contains only one element. 

When $n=3$ we will speak about a {\it triple} Massey product.

For $n\geq2$ we say that $\sC^\bullet$ has the {\it vanishing $n$-fold Massey product property}
if every defined Massey product $\langle a_1,\ldots, a_n\rangle$, where $ a_1,\ldots, a_n\in \sC^1$,
necessarily contains $0$.
\end{defn}

\begin{rmk}
\label{rmk:indet}
Let  $a_1,\ldots,a_n\in H^1$.
Suppose that $M=(a_{ij})$, $1\leq i<j\leq n+1$ and $(i,j)\not=(1,n+1)$,
is a defining system for the $n$-fold Massey product of $a_1,\ldots,a_n$.
It is straightforward to see that 
\[
\langle a_1,\ldots,a_n\rangle_M+ a_1\cup H^1+ H^1\cup a_n \subset \langle a_1,\ldots,a_n\rangle.
\]
And if $n=3$ then 
\[
\langle a_1,a_2,a_3\rangle_M+ a_1\cup H^1+ H^1\cup a_3 = \langle a_1,a_2,a_3\rangle. 
\]
In particular, $\langle a_1,a_2,a_3\rangle$ is uniquely defined if and only if $a_1\cup H^1= H^1\cup a_3=0$.
\end{rmk}

\section{Massey products and unipotent representations}
\label{sec:3}
Let  $G$ be a profinite group and let $A$ be a finite commutative ring considered as a trivial discrete $G$-module.
The  complex $\sC^\bullet=(\sC^\bullet(G,A),\partial)$ of inhomogeneous continuous cochains of $G$
with coefficients in $A$, is a DGA with the cup product  \cite[Ch.\ I, \S2 and Proposition 1.4.1]{NSW}. (Technically \cite[Proposition 1.4.1]{NSW} deals with homogeneous continuous cochains. However it is straightforward to see, using this proposition and the relationship between homogeneous and inhomogeneous
continuous cochains in  \cite[Ch. I, \S2]{NSW}, that this proposition is true also for inhomogeneous continuous cochains.)
We write $H^i(G,A)$ for the corresponding cohomology groups.
As observed by Dwyer \cite{Dwy75} in the discrete context (see also \cite[\S8]{Ef} in the profinite case),
defining systems for this DGA can be interpreted in terms of upper-triangular unipotent representations of $G$, as follows.

Let $n\geq 3$ be an integer.  Let $\U_{n+1}(A)$ be the group of all upper-triangular unipotent $(n+1)\times(n+1)$-matrices
with entries in  $A$.
Let $Z_{n+1}(A)$ be the subgroup of all such matrices with all off-diagonal entries
being $0$ except at position $(1,n+1)$. This group is the center of $\U_{n+1}(A)$.
We may identify $\U_{n+1}(A)/Z_{n+1}(A)$ with the group $\bar\U_{n+1}(A)$
of all upper-triangular unipotent $(n+1)\times(n+1)$-matrices with entries over $A$
with the $(1,n+1)$-entry omitted.

For a representation $\rho\colon G\to \U_{n+1}(A)$ and $1\leq i< j\leq n+1$
let $\rho_{ij}\colon G\to A$
be the composition of $\rho$ with the projection from $\U_{n+1}(A)$ to its $(i,j)$-coordinate.
We use  similar notation for representations $\bar\rho\colon G\to\bar\U_{n+1}(A)$.
Note that $\rho_{i,i+1}$ (resp., $\bar\rho_{i,i+1}$) is a group homomorphism.

\begin{thm}[{\cite[Theorem 2.4]{Dwy75}}]
\label{thm:Dwyer}
Let $\alpha_1,\ldots,\alpha_n$ be elements of $H^1(G,A)$. There is a one-one correspondence $M\leftrightarrow \bar\rho_M$ between defining systems $M$ for $\langle \alpha_1,\ldots,\alpha_n\rangle$ and group homomorphisms $\bar\rho_M:G\to \bar{\U}_{n+1}(A)$ with $(\bar\rho_M)_{i,i+1}= -\alpha_i$, for $1\leq i\leq n$.

Moreover $\langle \alpha_1,\ldots,\alpha_n\rangle_M=0$ in $H^2(G,A)$ if and only if the dotted arrow exists in the following  commutative diagram
\[
\xymatrix{
& & &G \ar@{->}[d]^{\bar\rho_M} \ar@{-->}[ld]\\
0\ar[r]& A\ar[r] &\U_{n+1}(A)\ar[r] &\bar{\U}_{n+1}(A)\ar[r] &1.
}
\]
\end{thm}
Explicitly, the one-one correspondence in Theorem~\ref{thm:Dwyer} is given by: For a defining system $M=(a_{ij})$ for $\langle \alpha_1,\ldots,\alpha_n\rangle$, $\bar\rho_M\colon G\to \bar{\U}_{n+1}(A)$ is given by letting $(\bar\rho_M)_{ij}=-a_{ij}$. 
\begin{rmk}
\label{rmk:G and G(p)}
Let $G(p)$ be the maximal pro-$p$-quotient of $G$. Then the natural map
 \[\pi^*\colon H^1(G(p),\F_p)\to H^1(G,\F_p),\] 
 which is induced from the  quotient map $\pi\colon G\to G(p)$,  is an isomorphism. 
%Let $\pi^*\colon H^\bullet(G(p),\F_p)\to H^\bullet(G,p)$ be the canonical map induced from the quotient map $\pi\colon G\to G(p)$. 
Let  $\alpha_1,\ldots,\alpha_n$ be elements of $H^1(G(p),\F_p)$. Then the following are equivalent:
\begin{enumerate}
\item The $n$-fold Massey product $\langle \alpha_1,\ldots,\alpha_n\rangle$ is defined and contains 0 in $H^2(G(p),\F_p)$
\item The $n$-fold Massey product $\langle \pi^*(\alpha_1),\ldots,\pi^*(\alpha_n)\rangle$ is defined and contains $0$ in $H^2(G,\F_p)$.
\end{enumerate}
This follows from Theorem~\ref{thm:Dwyer} and from the fact that $\U_{n+1}(\F_p)$ and $\bar{\U}_{n+1}(\F_p)$ are (finite) $p$-groups.
\end{rmk}
\begin{defn}
\label{defn:vanishing n-Massey}
Let the notation be as above. We say that $G$ has the {\it vanishing $n$-fold Massey product property (with respect to $A$)} if the DGA $\sC^\bullet(G,A)$ has the vanishing $n$-fold Massey product property.
\end{defn}

\begin{cor}
\label{cor:vanishing-lifting} Let $n\geq 3$ be an integer. 
The following conditions are equivalent.
\begin{enumerate}
\item[(i)]
$G$ has the vanishing $n$-fold Massey product property with respect to $A$.
\item[(ii)]
For every representation $\bar{\rho}\colon G\to \bar\U_{n+1}(A)$, there is a representation
$\rho\colon G\to\U_{n+1}(A)$ such that $\rho_{i,i+1}=\bar{\rho}_{i,i+1}$,
for $i=1,2,\ldots, n$.
\end{enumerate}
\end{cor}

\begin{cor} 
\label{cor:G and G(p)}
Let $n\geq 3$ be an integer. Let $G(p)$ be the maximal pro-$p$ quotient of $G$ and assume that $A=\F_p$. Then the following conditions are equivalent.
\begin{enumerate}
\item[(i)]
$G$ has the vanishing $n$-fold Massey product property with respect to $\F_p$.
\item[(ii)]
$G(p)$ has the vanishing $n$-fold Massey product property with respect to $\F_p$.
\end{enumerate}
\end{cor}

\begin{prop}[{\cite[p. 182, Remark]{Dwy75}, see also \cite[Proposition 8.3]{Ef}}] 
\label{prop:Dwyer}
Let $\bar\rho_M:G\to \bar\U_{n+1}(A)$ correspond to a defining system $M=(c_{ij})$ for $\langle \alpha_1,\ldots,\alpha_n\rangle$ as in Theorem~\ref{thm:Dwyer}. Then the central extension associated with $\langle \alpha_1,\ldots,\alpha_n\rangle_M$ is the pull-back
\[
0\to A \to \U_{n+1}(A)\times_{\bar\U_{n+1}(A)} G \to G \to 1
\]
via $\bar\rho_M:G\to \bar\U_{n+1}(A)$ of the extension
\[0\to A \to \U_{n+1}(A)\to \bar \U_{n+1}(A) \to 1.\]
\end{prop}

Now assume that $G=S/R$ is the quotient of some profinite group $S$ by some normal subgroup $R$. Then we have the transgression map (\cite[Chapter I, Prop 1.6.6]{NSW})
\[
\trg: H^1(R,A)^G\to H^2(G,A).
\]
Let  $\bar\rho: G\to \bar\U_{n+1}(A)$ be a representation of $G$ and let $\langle -\bar\rho_{12},\ldots, -\bar\rho_{n,n+1}\rangle_{\bar\rho}$ be the $n$-fold Massey product value relative to the defining system corresponding to $\bar\rho$. 
Suppose that $\rho:S\to \U_{n+1}(A)$ is a {\it lift} of $\bar\rho$, i.e., $\rho$ is a homomorphism such that the diagram
\[
\xymatrix{
S \ar@{->}[d]^{\rho} \ar@{->}[r] &G \ar@{->}[d]^{\bar\rho}\ar@{->}[r] &1\\
\U_{n+1}(A)\ar[r] &\bar{\U}_{n+1}(A)\ar[r] &1
}
\]
commutes. 
We can define (see \cite[page 8]{Sha}) $\Lambda(\rho)\in H^1(R,A)^{G}$ by
\[
\Lambda(\rho)(\tau)=-\rho_{1,n+1}(\tau), 
\]
for $\tau\in R$. Then by the same argument as in \cite[Lemma 2.3]{Sha} and by Proposition~\ref{prop:Dwyer}, we obtain the following result. We include a proof here for the convenience of the reader.
\begin{lem}
\label{lem:Sharifi}
We have $\trg(\Lambda(\rho))=\langle -\bar\rho_{12},\ldots, -\bar\rho_{n,n+1}\rangle_{\bar\rho}$.
\end{lem}
\begin{proof}
We consider the following diagram
\[
\xymatrix{
1\ar@{->}[r]  &R \ar@{->}[r] \ar@{->}[d]^{-\Lambda(\rho)} &S \ar@{->}[r]  \ar@{->}[d] &G \ar@{->}[r]\ar@{=}[d] &1\\
0\ar@{->}[r] &A \ar@{->}[r] \ar@{=}[d] & \sE \ar@{->}[r] \ar@{-->}[d] &G \ar@{->}[r] \ar@{=}[d] &1\\
0\ar@{->}[r] &A \ar@{->}[r]  \ar@{=}[d] &\U_{n+1}(A)\times_{\bar\U_{n+1}(A)} G \ar@{->}[r]  \ar@{->}[d] &G \ar@{->}[r]  \ar@{->}[d]_{\bar\rho}&1\\
0\ar@{->}[r] &A \ar@{->}[r]   &\U_{n+1}(A) \ar@{->}[r] &\bar\U_{n+1}(A) \ar@{->}[r] &1
}
\]
and we read the diagram from the top to the bottom. 
Here the second exact sequence is the pushout of the first exact sequence via $\Lambda(\rho):R\to A$. Then its equivalence class as an element in $H^2(G,A)$ is $\trg(\Lambda(\rho))$. 

 On the other hand, by Proposition~\ref{prop:Dwyer} the equivalence class of the third central extension in $H^2(G,A)$ is $\langle -\bar\rho_{12},\ldots, -\bar\rho_{n,n+1}\rangle_{\bar\rho}$. In order to prove the lemma, we only need to prove that there exists a dashed arrow $\sE \dashrightarrow \U_{n+1}(A)\times_{\bar\U_{n+1}(A)} G$ making the above diagram commute. But this  follows from the universal properties of the pullback $\U_{n+1}(A)\times_{\bar\U_{n+1}(A)} G$ and the pushout $\sE$.
\end{proof}

Now let $A=\F_p$, with $p$ a prime number.
As shown for example in \cite{Ef,Ga,Mor,Vo}, Massey products in $\sC^\bullet(G,\F_p)$ are also intimately related to the
{\it $p$-Zassenhaus filtration} $G_{(n)}$, $n=1,2,\ldots$ of $G$.
Recall that this filtration is defined inductively by
\[
G_{(1)}=G, \quad G_{(n)}=G_{(\lceil n/p\rceil)}^p\prod_{i+j=n}[G_{(i)},G_{(j)}],
\]
where $\lceil n/p \rceil$ is the least integer which is greater than or equal to $n/p$.

\begin{lem}
\label{lem:trivial on Zassenhauss subgroup}
 Let $G$ be a profinite group. Then 
\begin{enumerate}
\item Every (continuous) homomorphism $\rho: G\to \U_{n+1}(\F_p)$ is trivial on $G_{(n+1)}$.
\item Every (continuous) homomorphism $\rho: G\to \bar\U_{n+1}(\F_p)$ is trivial on $G_{(n+1)}$.
\end{enumerate}
\end{lem}
\begin{proof} These follow from the fact that $\U_{n+1}(\F_p)_{(n+1)}=1$.
\end{proof}

\begin{lem}
\label{lem:G mod Zassenhaus subgroup}
The profinite group $G$ has the vanishing $n$-fold Massey product property with respect to $\F_p$ if and only if $G/G_{(n+1)}$ has this property also.
\end{lem}
\begin{proof}
This follows from Corollary~\ref{cor:vanishing-lifting} and Lemma~\ref{lem:trivial on Zassenhauss subgroup}.
\end{proof}

\begin{prop}
\label{prop:equiv}
Let $N,N^\prime$ be closed normal subgroups of a free pro-$p$-group $S$ such that $NS_{(n+1)}=N^\prime S_{(n+1)}$.
Then $G=S/N$ has the vanishing $n$-fold Massey product property with respect to $\F_p$ if and only if
$G^\prime=S/N^\prime$ has the vanishing $n$-fold Massey product property with respect to $\F_p$.
\end{prop}
\begin{proof}
Because surjective homomorphisms take $n$th $p$-Zassenhaus filtrations onto $n$th $p$-Zassenhaus filtrations,  using our assumption, we have
\[
G/G_{(n+1)}\cong S/NS_{(n+1)}=S/N'S_{(n+1)}\cong G^\prime/G^\prime_{(n+1)}.
\]
Therefore our result follows from Lemma~\ref{lem:G mod Zassenhaus subgroup}.
\end{proof}

\section{First examples}

\begin{ex}
\label{ex:free}
If $G$ is a free pro-$p$-group, then it has the $n$-fold Massey product vanishing property for every $n\geq2$ because $H^2(G,\F_p)=0$.
Alternatively, this follows from the universal property of $G$ and condition (ii) of Corollary \ref{cor:vanishing-lifting}.
\end{ex}
Recall that a profinite group $G$ is projective (in the category of profinite groups)  if for any finite groups $A$ and $B$, and for any surjective morphisms $\rho\colon G \to A$, $\alpha\colon B\to A$, there exists a homomorphism $\gamma\colon G\to B$ such that $\rho=\gamma\circ \alpha$. (See \cite[page 207]{FJ}.)
\begin{ex}
\label{ex:projective}
Let $G$ be a projective group, then it has the $n$-fold Massey product vanishing property for every $n\geq 3$ and for every $p$. This follows directly from  the definition of projective groups and  condition (ii) of Corollary \ref{cor:vanishing-lifting}.
\end{ex}

A pro-$p$-group $G$ is said to be a Demushkin group if
\begin{enumerate}
\item $\dim_{\F_p} H^1(G,\F_p)<\infty,$ 
\item $\dim_{\F_p} H^2(G,\F_p)=1,$
\item  the cup product $H^1(G,\F_p)\times H^1(G,\F_p)\to H^2(G,\F_p)$ is a non-degenerate bilinear form.
\end{enumerate}

\begin{thm}
\label{thm:Demushkin}
Let $n\geq 3$ be an integer and $p$ a prime number. Then every pro-$p$ Demushkin group has  the vanishing $n$-fold Massey product property with respect to $\F_p$.
\end{thm}

The following proof is adapted from that of \cite[Lemma 3.5]{HW}.
\begin{proof}
Let $G$ be a pro-$p$  Demushkin group. Let $\chi_1,\ldots,\chi_n$ be in $H^1(G,\F_p)$. Assume that the triple Massey product $\langle \chi_1,\ldots,\chi_n\rangle$ is defined. If $\chi_1=0$ then by \cite[Lemma 6.2.4]{Fe}, which is valid in the profinite case as well,  $\langle \chi_1,\ldots,\chi_n\rangle$ contains 0. So we may assume that $\chi_1\not=0$. In this case, to show that $\langle \chi_1,\ldots,\chi_n\rangle$ contains 0, we only need to show that
\[
\chi_1\cup (-): H^1(G,\F_p)\to H^2(G,\F_p)
\]
is surjective by Remark~\ref{rmk:indet}. From the definition of Demushkin groups, one has 
\[H^2(G,\F_p)\simeq \F_p.\]
 So it is enough to show that the map $\chi_1\cup (-) $ is non-zero. But this follows from 
 the non-degenerate property of the cup product 
$H^1(G,\F_p)\times H^1(G,\F_p)\to H^2(G,\F_p).$
\end{proof}

\begin{rmk} If $F$ is a finite field extension of $\Q_p$ containing a primitive $p$-th root of unity, then $G_F(p)$ is a pro-$p$ Demushkin group.  In \cite{Sha}, Shafarevich showed that if $F$ is as above, but $F$ does not contain any primitive $p$-th root of unity, then $G_F(p)$ is a free pro-$p$-group.

Demushkin groups along with free pro-$p$-groups, abelian torsion free pro-$p$-groups, and cyclic groups of order 2 play a dominant role in the current investigation of finitely generated subgroups of maximal pro-$p$ quotients $G_F(p)$ of absolute Galois groups. The elementary conjecture predicts that these groups above are all "building blocks" for $G_F(p)$. (See \cite{Ef1, Mar, LLMS, JW}.)
\end{rmk}

\begin{prop}
\label{prop:free product}
Let $G_1,G_2$  be two pro-$p$-groups. Then the free pro-$p$ product $G_1*G_2$  has the vanishing $n$-fold Massey product property with respect to $\F_p$ if and only if both $G_1$ and $G_2$ have this property as well.
\end{prop}
\begin{proof} Assume that $G_1,G_2$ have  the vanishing $n$-fold Massey product property. 
Let $\bar\rho\colon G_1*G_2\to \bar\U_{n+1}(\F_p)$ be any homomorphism.
By Corollary \ref{cor:vanishing-lifting}, we need to find a homomorphism $\rho\colon G_1*G_2\to\U_{n+1}(\F_p)$
such that $\rho_{j,j+1}=\bar\rho_{j,j+1}$, $j=1,2,\ldots, n$.

For each $i=1,2$ let $\kappa_i\colon G_i\to G_1*G_2$ be the natural monomorphism, and set $\bar\rho_i=\bar\rho\circ\kappa_i$.
Since $G_i$ has the vanishing $n$-fold Massey product property, there is a representation
$\rho_i\colon G_i\to \U_{n+1}(\F_p)$ such that $(\rho_i)_{j,j+1}= ( \bar\rho_i)_{j,j+1}$, for $j=1,2,\ldots, n$.
The universal property of free products yields a unique homomorphism $\rho \colon G_1*G_2\to \U_{n+1}(\F_p)$
such that $\rho\circ\kappa_i=\rho_i$, $i=1,2$.
For $i=1,2$ and $ j=1,2,\ldots,n$ we have
\[
\rho_{j,j+1}\circ\kappa_i=(\rho\circ\kappa_i)_{j,j+1}
= (\rho_i)_{j,j+1}= (\bar\rho_i)_{j,j+1}\\
= (\bar\rho\circ\kappa_i)_{j,j+1}=\bar\rho_{j,j+1}\circ\kappa_i,
\]
so $\rho_{j,j+1}=\bar\rho_{j,j+1}$, as desired.

Conversely, assume that $G_1* G_2$ has the vanishing $n$-fold Massey product property. Let $\bar\rho_1:G_1\to \bar\U_{n+1}(\F_p)$ be any representation of $G_1$. Let $\bar{\rho}_2:G_2\to \bar{\U}_{n+1}(\F_p)$ be the trivial homomorphism. Then by the universal property of free products, there exists a homomorphism $\bar\rho:G_1*G_2\to \bar\U_{n+1}(\F_p)$ such that $\bar\rho_1=\bar\rho\circ \kappa_1$. Since $G_1* G_2$ has the vanishing $n$-fold Massey product property, there exists a homomorphism $\rho\colon G_1*G_2\to \U_{n+1}(\F_p)$
 such that $\rho_{j,j+1}=\bar\rho_{i,i+1}$ for $i=1,2,\ldots,n$. Let $\rho_1:G_1\to \U_{n+1}(\F_p)$ be the composite $\rho\circ\kappa_1$. Then for $i=1,2,\ldots,n$, we have
 \[
 (\rho_1)_{i,i+1}= (\rho\circ\kappa_1)_{i,i+1}= \rho_{i,i+1}\circ\kappa_1= \bar\rho_{i,i+1}\circ\kappa_1= (\bar\rho\circ \kappa_1)_{i,i+1}=(\bar\rho_1)_{i,i+1}.
 \]
 Hence by Corollary \ref{cor:vanishing-lifting}, $G_1$ has the vanishing $n$-fold Massey product property. 

Similarly, $G_2$ has the vanishing $n$-fold Massey product property.
  \end{proof}

 Let $p>2$ be an odd prime and $G$  a pro-$p$-group. Let $\chi$ be an element in $H^1(G,\F_p)$. In \cite[Section 3]{Kra}, Kraines defined a restricted $n$-fold Massey product $\langle \chi\rangle^n$. If a restricted $n$-fold Massey product $\langle\chi\rangle^n$is defined then so is the $n$-fold Massey product $\langle\chi,\ldots,\chi\rangle$, and the latter contains the former. Kraines showed that  $\langle \chi\rangle^n=0$ for $n=2,\ldots,p-1$ and $\langle \chi\rangle^p$ is defined (\cite[Theorem 15]{Kra}). In fact
 $\langle \chi\rangle^p=-\beta(\chi)$, where $\beta: H^1(G,\F_p)\to H^2(G,\F_p)$ is the Bockstein homomorphism, i.e., the connecting homomorphism induced by the exact sequence
\[
0\to \Z/p \to \Z/p^2 \to \Z/p \to 0.
\] 
Using Kraines' results mentioned above, we obtain the following result. 
 \begin{prop}
 \label{prop:p-Massey} 
Let $n$ be an integer with $2<n\leq p$.  Let $F$ be any field containing a primitive $p$-th root of unity if ${\rm char}(F)\not=p$. Let $G$ be the absolute Galois group $G_F$ of $F$ or its maximal pro-$p$ quotient $G_F(p)$. Then for any  $\chi \in H^1(G,\F_p)$, the $n$-fold Massey product $\langle \chi,\ldots,\chi\rangle$ is defined and contains 0.
 %Let $n$ be an integer with $2<n\leq p$. Then for any $a$ in $F^*$, the $n$-fold Massey product $\langle \chi_a,\ldots,\chi_a\rangle$ contains 0. 
\end{prop}
\begin{proof}
It is enough to consider the case $G=G_F(p)$ by Remark~\ref{rmk:G and G(p)}. Also if ${\rm char} F=p$ then since $G_F(p)$ is a free pro-$p$-group, $\langle \chi,\ldots,\chi\rangle=0$. 
So we may assume that ${\rm char} F\not=p$ and let us fix a primitive $p$-th root of unity $\xi$. 
Then $\chi=\chi_a$ for some $a\in F^*$, 
where $\chi_a\in H^1(G_F,\F_p)= H^1(G_F(p),\F_p)$ is  the  character associated to $a$ via the Kummer map $F^*\to H^1(G_F,\F_p)=H^1(G_F(p),\F_p)$.

If $n<p$ then by \cite[Theorem 15]{Kra} $\langle \chi_a,\ldots,\chi_a\rangle$ contains $0=\langle \chi_a\rangle^n.$ 

Now we consider the case $n=p$.  Then $\langle \chi_a \rangle^p=-\beta(\chi_a)$. 
By \cite[proof of Proposition 3.2]{EM1}, $\beta(\chi_a)=\chi_a\cup \chi_\xi$ ($\xi\in F^*$ is a fixed primitive $p$-th root of unity). Hence by Remark~\ref{rmk:indet}, one has
\[0= \langle \chi_a\rangle^p +\chi_a\cup\chi_\xi \in \langle \chi_a,\ldots,\chi_a\rangle,
\]
as claimed.
\end{proof} 

\begin{ex}
\label{ex:p-cyclic}
 Let $p$ be an odd prime number and $G=\Z/p\Z$. Let $\chi\in H^1(G,\F_p)$ be the identity map. Then the $p$-fold Massey product $\langle \chi,\ldots,\chi\rangle$ is defined but does not contain 0. Suppose that, contrarily,  the $p$-fold Massey product $\langle\chi,\ldots,\chi\rangle$ contains 0, then there exists a representation $\rho\colon G \to \U_{p+1}(\F_p)$ such that $\rho_{i,i+1}=\chi$, for $i=1,\ldots,p$. Let $B:=\rho(\bar{1})\in \U_{p+1}(\F_p)$. Then all entries of $B$ at positions $(i,i+1)$, $i=1,\ldots, p$, are equal to 1. Hence $B^p\not=1$, and this contradicts  the fact that $B$ is the image of an element of order $p$.
\end{ex}
\begin{rmk} 
\label{rmk:Artin-Schreier}
Proposition~\ref{prop:p-Massey} and Example~\ref{ex:p-cyclic} immediately provide an explanation to a part of the well-known Artin-Schreier theorem \cite{AS1,AS2}  (respectively, Becker's theorem \cite{Be}) which says that the absolute Galois group $G_F$ (respectively, its maximal pro-$p$-quotient $G_F(p)$) of any field $F$ cannot have an element of odd prime order. (Note also that if $G_F\simeq \Z/p\Z$ then $F$ contains a primitive $p$-th root of unity.)

In \cite{MT2}, using Galois automatic realization of given groups, we shall prove a more general result  than Proposition~\ref{prop:p-Massey} in which the condition $n\leq p$ can be omitted, provided that if $p=2$ then $-1$ is a square in $F$. One  can then use this generalized result to show the full Artin-Schreier theorem (respectively, Becker's theorem). (See \cite{MT2}.)
\end{rmk}

\section{Splitting variety and the vanishing property}
Let $F$ be a field of characteristic $\not=2$. Let $G=G_{F}(2)$ be the maximal pro-2 Galois group of $F$.
Let $a,b,c$ be elements in $F^*$ and $\chi_a,\chi_b,\chi_c\in H^1(G,\F_2)$ be the characters corresponding to $a,b,c$ via the Kummer map $F^*\to H^1(G,\F_2)$. Let $X_{a,b,c}$ be the variety in $\G_m\times \A^4$ defined by the equation
\[
bX^2= (Y_1^2-aY_2^2+cY_3^2-acY_4^2)^2-4c(Y_1Y_3-aY_2Y_4)^2.
\]

\begin{proof}[First proof of Theorem~\ref{thm:vanishing}]
If $a$ (or $b$, or $c$) is in $(F^*)^2$ then the corresponding character $\chi_a$ (or $\chi_b$, or $\chi_c$) is the trivial character and hence the Massey product $\langle \chi_a,\chi_b,\chi_c\rangle$ contains 0 by \cite[Lemma 6.2.4]{Fe}. So we may assume that $a,b$ and $c$ are not in $(F^*)^2$. The following well-known fact will be used frequently: $\chi_a\cup\chi_b=0$ if and only if $b$ is in $N_{F(\sqrt{a})/F}(F(\sqrt{a})^*)$ (see e.g., \cite[Introduction, p. 4]{HW}, \cite[Chapter XIV, Propositions 4-5]{Se2}, or \cite[Lemma 8.4]{Sri}). 
There are two cases to consider.
\\
\\
{\bf Case 1:} $a/c$ is in $(F^*)^2$. Then $\chi_a=\chi_c$ and hence $\langle \chi_a,\chi_b,\chi_c\rangle =\langle \chi_a,\chi_b,\chi_a\rangle$ and we can assume $a=c$. 
Since $\langle \chi_a,\chi_b,\chi_a\rangle$ is defined, $\chi_a\cup\chi_b=0$. Hence $b\in N_{F(\sqrt{a})/F}(F(\sqrt{a})^*)$ and there exists $\alpha_1,\alpha_2\in F$ such that 
\[
b= N_{F(\sqrt{a})/F}(\alpha_1+\alpha_2\sqrt{a})=\alpha_1^2-a\alpha_2^2.
\]
If $\alpha_1\not=0$ then  let $ x=4\alpha_1\not=0, y_1=2\alpha_1,y_2=y_3=\alpha_2,y_4=0.$ One has
\[
\begin{aligned}
(y_1^2-ay_2^2+ay_3^2-a^2y_4^2)^2-4a (y_1y_3-ay_2y_4)^2=(4\alpha_1^2)^2- 4a(2\alpha_1\alpha_2)^2=16\alpha_1^2(\alpha_1^2-a \alpha_2^2)
= bx^2.
\end{aligned}
\]
If $\alpha_1=0$ then $b=-a\alpha_2^2$. Let $ x=4a\not=0, y_1=a,y_2=y_3=\alpha_2,y_4=-1.$ Then one has
\[
\begin{aligned}
(y_1^2-ay_2^2+ay_3^2-a^2y_4^2)^2-4a (y_1y_3-ay_2y_4)^2=0- 4a(2a\alpha_2)^2=bx^2.\\
\end{aligned}
\]
\\
{\bf Case 2:} $a/c$ is not in $(F^*)^2$. Since $\langle \chi_a,\chi_b,\chi_c\rangle$ is defined, $\chi_a\cup\chi_b=0=\chi_b \cup \chi_c$. Hence $b\in N_{F(\sqrt{a})/F}(F(\sqrt{a})^*)$ and $b\in N_{F(\sqrt{c})/F}(F(\sqrt{c})^*)$. Thus, there exists $\alpha_1,\alpha_2,\gamma_1,\gamma_2\in F$ such that
\[
\begin{aligned}
b&= N_{F(\sqrt{a})/F}(\alpha_1+\alpha_2\sqrt{a})&=\alpha_1^2-a\alpha_2^2\\
&= N_{F(\sqrt{c})/F}(\gamma_1+\gamma_2\sqrt{c})&=\gamma_1^2-c\gamma_2^2.
\end{aligned}
\]
Hence $c\gamma_2^2- a\alpha_2^2=\gamma_1^2-\alpha_1^2\not=0$ because $a/c$ and $b$ are not in $(F^*)^2$. Therefore $\alpha_1+\gamma_1\not=0$.

Let 
\[
x=2(\alpha_1+\gamma_1), y_1=\alpha_1+\gamma_1,y_2=\alpha_2,y_3=\gamma_2,y_4=0.
\]
Then 
\[
\begin{aligned}
(y_1^2-ay_2^2+cy_3^2-acy_4^2)^2-4c (y_1y_3-ay_2y_4)^2&=[(\alpha_1+\gamma_1)^2-a\alpha_2^2+c\gamma_2^2]^2- 4c[(\alpha_1+\gamma_1)\gamma_2]^2\\
&=[(\alpha_1+\gamma_1)^2+\gamma_1^2-\alpha_1^2]^2-4c(\alpha_1+\gamma_1)^2\gamma_2^2\\
&=4(\alpha_1+\gamma_1)^2\gamma_1^2 -4c (\alpha_1+\gamma_1)^2\gamma_2^2\\
&= 4 (\alpha_1+\gamma_1)^2(\gamma_1^2-c\gamma_2^2)\\
&=4(\alpha_1+\gamma_1)^2b= bx^2.
\end{aligned}
\]
Therefore the variety $X_{a,b,c}$ contains an $F$-rational point, namely $(x,y_1,y_2,y_3,y_4)$. Hence $\langle \chi_a,\chi_b,\chi_c\rangle$ contains 0 by \cite[Corollary 2.7]{HW}.
\end{proof} 

%\begin{rmk} Theorem~\ref{thm:vanishing} is proved in \cite{HW} for global fields of characteristic $\not=2$. \end{rmk}
\section{Field theory and vanishing property}
\label{sec:2nd proof}

In this section we present another approach to prove Theorem~\ref{thm:vanishing} using Galois theory and \cite{GLMS}. 

Notation: For $a,b$ in a field $F$ of characteristic $\not=2$, $(a,b)_F$ is the quaternion algebra generated by $a$ and $b$. For $x,y,z$ in a group, $[x,y]=x^{-1}y^{-1}xy$.

\begin{proof}[Second proof of Theorem~\ref{thm:vanishing}]
 As in the first proof, we may assume that $a,b$ and $c$ are not in $(F^*)^2$.
 
Assume that the triple Massey product $\langle \chi_b,\chi_a,\chi_c\rangle$ is defined, we show that it contains 0. (Note that  the order in the triple Massey product here is different from the one in the first proof, because we want to be consistent with the notation in \cite{GLMS}.)
\\
\\
{\bf Case 1:} $a\equiv b\equiv c \bmod (F^*)^2$. Then $\langle \chi_b,\chi_a,\chi_c\rangle =\langle \chi_b,\chi_b,\chi_b\rangle$. Since $(b,b)_F=0$, $b$ is a norm of $F(\sqrt{b})/F$, i.e., $b=N_{F(\sqrt{b})/F}(\beta)$ for some $\beta \in F(\sqrt{b})$. Let $L=F(\sqrt{\beta})$, then $L/F$ is a Galois extension which is cyclic of order 4. Its Galois group is generated by  $\sigma_b\in \Gal(L/F)$, where $ \sigma_b(\sqrt{\beta})=\sqrt{b}/\sqrt{\beta}$.

One has the following homomorphism $\varphi:\Gal(L/F)\to \U_4(\F_2)$ by letting
\[
\sigma_b\mapsto B:=\begin{bmatrix}
1& 1 & 0 & 0\\
0& 1 & 1 & 0\\
0& 0 & 1 & 1\\
0& 0 & 0 & 1
\end{bmatrix}.
\]

Let $\rho$ be the  composite homomorphism $\rho:\Gal_F\to \Gal(L/F)\stackrel{\varphi}{\to} \U_4(\F_2)$. Then one can check that
\[
\rho_{i,i+1}=\chi_b\quad \forall i=1,2,3.
\]
(Note that $\sigma_b|_{F(\sqrt{b})/F}$ maps $\sqrt{b}$ to $-\sqrt{b}$ and here we are identifying $\F_2=\{-1,1\}=\{0,1\}$.)
Hence by Theorem~\ref{thm:Dwyer}, the triple Massey product $\langle \chi_b,\chi_b,\chi_b\rangle$ contains 0.
\\
\\
{\bf Case 2:} $a\equiv b \bmod (F^*)^2$ and $a\not \equiv c\bmod (F^*)^2$. This case can be treated in a similar way  to Case 3 below.
\\
\\
{\bf Case 3:} $a\not \equiv b\bmod (F^*)^2$ and $c\equiv a\bmod (F^*)^2$. Then $\langle \chi_b,\chi_a,\chi_c\rangle =\langle \chi_b,\chi_a,\chi_a\rangle$. Since $(a,a)_F=(a,b)_F=0$ in the Brauer group ${\rm Br}(F)$, by construction in \cite[Section 3]{GLMS}, we have a Galois extension $L/F$ which contains $F(\sqrt{a},\sqrt{b})$ with Galois group $G_1$ described below. Also there exist $\sigma_a,\sigma_b\in \Gal(L/F)$ such that 
\[ \sigma_a(\sqrt{a})=-\sqrt{a}, \sigma_a(\sqrt{b})=\sqrt{b}, \sigma_b(\sqrt{a})=\sqrt{a},\sigma_a(\sqrt{b})=-\sqrt{b}.\]
 Let $G_1$ be the group generated by two symbols $x,y$ subject to the relations: $x^4=y^2=1=(x,y)^2=(x,y,x)^2$ and $(x,y,x)$ commutes with $x$ and $y$. Then it is shown in \cite{GLMS} that $\sigma_a,\sigma_b$ generates $\Gal(L/F)$ and $\Gal(L/F)$ is isomorphic to $G_1$ by letting $\sigma_a\mapsto x$ and $\sigma_b\mapsto y$.

Let 
\[
 u:=\begin{bmatrix}
1& 0 & 0 & 0\\
0& 1 & 1 & 0\\
0& 0 & 1 & 1\\
0& 0 & 0 & 1
\end{bmatrix}, 
  v:=\begin{bmatrix}
1& 1 & 0 & 0\\
0& 1 & 0 & 0\\
0& 0 & 1 & 0\\
0& 0 & 0 & 1
\end{bmatrix}.
\]
Then $u^4=v^2=[u,v]^2=1$ and $[[u,v],u]$ is central and of order 2 in $\U_4(\F_2)$. Hence one has the  homomorphism $\varphi:\Gal(L/F)\to \U_4(\F_2)$ by letting $\sigma_a\mapsto u, \sigma_b\mapsto v$. (The homomorphism $\varphi$ is in fact injective so that $\varphi$ induces an isomorphism between $\Gal(L/F)$ and the subgroup generated by $u,v$. This follows from $Z(G_1)=\Z/2\Z$, which is the smallest non-trivial normal subgroup of $G_1$ and $[[u,v],u]\not=1$.)
 
Let $\rho$ be the  composite homomorphism $\rho:\Gal_F\to \Gal(L/F)\stackrel{\varphi}{\to} \U_4(\F_2)$. Then one can check that
\[
\rho_{12}=\chi_b \text{ and } \rho_{23}=\rho_{34}=\chi_a.
\]
Hence by Theorem~\ref{thm:Dwyer}, the triple Massey product $\langle \chi_b,\chi_a,\chi_a\rangle$ contains 0.
\\
\\
{\bf Case 4:} $a\not \equiv b\bmod (F^*)^2$ and $c\equiv b \bmod (F^*)^2$ . Then $\chi_b=\chi_c$ and hence $\langle \chi_b,\chi_a,\chi_c\rangle=\langle \chi_b,\chi_a,\chi_b\rangle$. By assumption, we have $(a,b)_F=0$. Hence $b=N_{F(\sqrt{a})/F}(\beta)$, for some $\beta\in F(\sqrt{a})$. Let $L=F(\sqrt{a},\sqrt{b},\sqrt{\beta})$. We define two automorphisms $\sigma_a, \sigma_b\in \Gal(L/F)$   as follows:
\[
\begin{aligned}
\sigma_a: &\sqrt{a}\mapsto -\sqrt{a}; \sqrt{b}\mapsto  \sqrt{b}; \sqrt{\beta}\mapsto -\sqrt{b}/\sqrt{\beta};\\
\sigma_b: &\sqrt{a}\mapsto \sqrt{a}; \sqrt{b}\mapsto -\sqrt{b}; \sqrt{\beta}\mapsto \sqrt{\beta}.\\
\end{aligned}
\]
Then the subgroup  of $\Gal(L/F)$ generated by $\sigma_a,\sigma_b$ is isomorphic to the dihedral group $D_4$ of order $8=[L:F]$. Hence $\Gal(L/F)$ is isomorphic to $D_4$ and generated by $\sigma_a,\sigma_b$.  One has the following homomorphism $\varphi:\Gal(L/F)\to \U_4(\F_2)$ by letting
\[
\sigma_a\mapsto u:=\begin{bmatrix}
1& 0 & 0 & 0\\
0& 1 & 1 & 0\\
0& 0 & 1 & 0\\
0& 0 & 0 & 1
\end{bmatrix},
\sigma_b\mapsto v:=
\begin{bmatrix}
1& 1 & 0 & 0\\
0& 1 & 0 & 0\\
0& 0 & 1 & 1\\
0& 0 & 0 & 1
\end{bmatrix}.
\]

Let $\rho$ be the  composite homomorphism $\rho:\Gal_F\to \Gal(L/F)\stackrel{\varphi}{\to} \U_4(\F_2)$. Then one can check that
\[
\rho_{23}=\chi_a \text{ and } \rho_{12}=\rho_{34}=\chi_b.
\]
Hence by Theorem~\ref{thm:Dwyer}, the triple Massey product $\langle \chi_b,\chi_a,\chi_b\rangle$ contains 0.
\\
\\
{\bf Case 5: } $a\not\equiv b\bmod (F^*)^2$ and $c\equiv ab\bmod (F^*)^2$.  Then $\langle \chi_b,\chi_a,\chi_c\rangle=\langle \chi_b,\chi_a,\chi_{ab}\rangle$.

By assumption, we have $(a,b)_F=(a,ab)_F=0$. Hence $(a,b)_F=(a,a)_F=0$. As in Case 3, we can construct a Galois extension $L/F$ with Galois group isomorphic to the group $G_1$.
Let  \[
A:=\begin{bmatrix}
1& 0 & 0 & 0\\
0& 1 & 1 & 0\\
0& 0 & 1 & 1\\
0& 0 & 0 & 1
\end{bmatrix},
 B:=
\begin{bmatrix}
1& 1 & 0 & 0\\
0& 1 & 0 & 0\\
0& 0 & 1 & 1\\
0& 0 & 0 & 1
\end{bmatrix}.
\]
Then $A^4=B^2=[A,B]^2$ and $[[A,B],A]$ is central and of order 2 in $\U_4(\F_2)$. Hence one has the following homomorphism $\varphi:\Gal(L/F)\to \U_4(\F_2)$ by letting $\sigma_a\mapsto A, \sigma_b\mapsto B$. Here $\sigma_a,\sigma_b$ as in Case 3. (The homomorphism $\varphi$ is in fact injective so that $\varphi$ induces an isomorphism between $\Gal(L/F)$ and the subgroup generated by $A,B$. This follows from $Z(G_1)=\Z/2\Z$ and $[[A,B],A]\not=1$.)

Let $\rho$ be the  composite homomorphism $\rho:\Gal_F\to \Gal(L/F)\stackrel{\varphi}{\to} \U_4(\F_2)$. Then one can check that
\[
\begin{aligned}
\rho_{12} =\chi_b, \; \rho_{23}=\chi_a,\; \rho_{34} =\chi_{ab}.
\end{aligned}
\]

Hence by Theorem~\ref{thm:Dwyer}, the triple Massey product $\langle \chi_b,\chi_a,\chi_{ab}\rangle$ contains 0.
\\
\\
{\bf Case 6}: $a,b,c$ are $\F_2$-independent in $F^*/{F^*}^2$.

Because $\langle \chi_b,\chi_a,\chi_c\rangle$ is defined, $(b,a)_F=(a,c)_F=0$. As in \cite{GLMS}, we have the following construction. 
There exist $\beta\in F(\sqrt{b})$ and $\gamma\in F(\sqrt{c})$ such that
$N_{F(\sqrt{b})/F}(\beta)=N_{F(\sqrt{c})/F}(\gamma)=a$. 
Let $E=F(\sqrt{b},\sqrt{c})$. Then \cite[Lemma 2.14]{Wad86} implies that there exist $\delta\in E$ and $d\in F$ such that $N_{E/F(\sqrt{b})}(\delta)=\beta d$ and $N_{E/F(\sqrt{c})}(\delta)=\gamma d$. Let $E^\prime=E(\sqrt{a})$, $K=E^\prime(\sqrt{\beta d},\sqrt{\gamma d})$ and $L=K(\sqrt{\delta})$. 
It is shown in \cite[Proof of Proposition 4.6]{GLMS} that there exist automorphisms $\sigma_a,\sigma_b,\sigma_c\in \Gal(L/F)$ such that $\sigma_{a}$ fixes $\sqrt{b},\sqrt{c}$ and $\sigma_a(\sqrt{a})=-\sqrt{a}$ and similarly for $\sigma_b,\sigma_c$. Furthermore, $\sigma_{a},\sigma_b,\sigma_c$ generates $\Gal(L/F)$. 

Let 
\[
X:=\begin{bmatrix}
1& 0 & 0 & 0\\
0& 1 & 1 & 0\\
0& 0 & 1 & 0\\
0& 0 & 0 & 1
\end{bmatrix}, 
 Y:=\begin{bmatrix}
1& 1 & 0 & 0\\
0& 1 & 0 & 0\\
0& 0 & 1 & 0\\
0& 0 & 0 & 1
\end{bmatrix}, 
 Z:=\begin{bmatrix}
1& 0 & 0 & 0\\
0& 1 & 0 & 0\\
0& 0 & 1 & 1\\
0& 0 & 0 & 1
\end{bmatrix}.
\]
By   direct computation    one has
\[
(1): X^2=Y^2=Z^2=1, [X,Y]^2=[X,Z]^2=[Y,Z]=1,
\]
and 
\[ (2):[Y,[X,Z]]=[Z,[X,Y]] \text{ is in the center and of order dividing 2}. 
\]
Hence there is a natural homomorphism $\varphi$ from $\Gal(L/F)\simeq G_2$  to $\U_4(\F_2)$ ($G_2$ is defined in \cite[Definition 4.4]{GLMS} as the group generated by $x,y,z$ satisfying two conditions (1)-(2) above).  As $X,Y,Z$ generates $\U_4(\F_2)$, $\varphi$ is surjective and hence an isomorphism because $|\Gal(L/F)|=|\U_4(\F_2)|=64$. Also from \cite[Proof of Proposition 4.7]{GLMS} one deduces that $\varphi$ maps $\sigma_a$ to $X$, $\sigma_b$ to $Y$ and $\sigma_c$ to $Z$.

Let $\rho$ be the following composite homomorphism $\rho:\Gal_F\to \Gal(L/F)\stackrel{\varphi}{\simeq} \U_4(\F_2)$. Then one can check that
\[
\begin{aligned}
\rho_{12} &=\chi_b,\; \rho_{23}=\chi_a, \; \rho_{34} =\chi_c.
\end{aligned}
\]

(Note that $\chi_a$ is the composition $\Gal_F\to \Gal(L/F)\to \Gal(F(\sqrt{a})/F)\simeq \F_2$, where the last map is the map sending $\sigma_a|_{F(\sqrt{a}/F)}$ to $1$ and similarly for $\chi_b,\chi_c$. Since  all the maps $\rho,\chi_a,\chi_b,\chi_c$ factor through $\Gal(L/F)$, it is enough to check on $\sigma_a,\sigma_b,\sigma_c$.)

Hence by Theorem~\ref{thm:Dwyer}, the triple Massey product $\langle \chi_b,\chi_a,\chi_c\rangle$ contains 0.
\end{proof}

\begin{rmk}
Because the Galois extensions $L/F$ with Galois group isomorphic with $\U_4(\F_2)$ play a fundamental role in the theory of triple Massey products, and for its use
in Galois theory, we shall describe the structure of these extensions. For further related results see \cite{GLMS} where $\U_4(\F_2)$ is called $G_2$. Let $X,Y,Z$ be matrices defined as in Case 6 of the previous proof. Then observe that
\[
\begin{aligned}
\U_4(\F_2)&=\{ X^\alpha Y^\beta Z^\gamma [X,Y]^\lambda [X,Z]^\mu [[X,Y],Z]^\nu\mid \alpha,\beta,\gamma,\lambda,\mu,\nu=0 \text{ or } 1\}\\
&= W \rtimes V,
\end{aligned}
\]
where $V$ is isomorphic to the Klein 4-group $V_4\simeq V = \langle Y,Z\rangle$ and $W\simeq \F_2[V]$.
% Because $H^2(V,W)=0$, $\U_4(\F_2)$ is in fact the only (up to isomorphism) extensioin of $V$ by $W$.

Now let $L/F$ be a $\U_4(\F_2)$-Galois extension. Let $E$ be the fixed field of $L$ by $W$. Then $E/F$ is a $V_4$-extension so  $E=F(\sqrt{b},\sqrt{c})$, for some $b,c\in F^*$ where $b,c$ are linearly independent mod $(F^*)^2$. 

Since $W$ is a 2-elementary group, $\Gal(L/E)=W$ and by Kummer theory one has $L=E(\sqrt{M})$, where $M\subset E^*/(E^*)^2$ is dual to $W$. 
Then $M$ is isomorphic to $\F_2[V]$. 

Let $[\delta]\in E^*/(E^*)^2$ be a generator of $M$. We define two automorphisms $\sigma_b, \sigma_b$  in $\Gal(E/F)$ as follows:  $\sigma_{b}(\sqrt{b})=-\sqrt{b}$, $\sigma_b(\sqrt{c})=\sqrt{c}$; $\sigma_{c}(\sqrt{b})=\sqrt{b}$, $\sigma_c(\sqrt{c})=-\sqrt{c}$.
 Then 
\[
\sigma_b(\delta)\delta=N_{E/F(\sqrt{c})}(\delta); \sigma_c(\delta)\delta=N_{E/F(\sqrt{b})}(\delta),
\]
and set
\[
a:=\sigma_b(\sigma_c(\delta))\sigma_c(\delta)\sigma_b(\delta)\delta=N_{E/F}(\delta).
\]
Now since $M$ is 4-dimensional we have $a\not\in (E^*)^2$. Hence we have shown that: 

Each $\U_4(\F_2)$-Galois extension $L/F$ is a normal closure of $E(\sqrt{\delta})/F$ where 
\begin{enumerate}
\item $E/F$ is  a $V_4$-extension.
\item  $\delta\in E^*$ and $N_{E/F}(\delta)\not \in (E^*)^2$.
\end{enumerate}

One can see that the converse for this also holds. Namely, if $L/F$ is a normal closure of $E(\sqrt{\delta})/F$ where $E/F$ and $\delta$ satisfy 2 conditions (1)-(2) above, then $L/F$ is a $\U_4(\F_2)$-Galois extension.
\end{rmk}

\begin{thm} 
\label{thm:Massey vanishing}
 Let $G$ be the absolute Galois group $G_F$ of a field $F$ or its maximal 2-extension quotient $G_F(2)$. Then $G$ has the  vanishing triple Massey product property with respect to $\F_2$.
\end{thm}
\begin{proof}
It is enough to consider the case $G=G_F(2)$ by Corollary~\ref{cor:G and G(p)}. 

If $F$ is of characteristic 2, then $G$ is free and hence $G$ has the vanishing triple Massey product property.

If $F$ is of characteristic $\not=2$, then $G$ has the vanishing triple Massey product property by Theorem~\ref{thm:vanishing}.
\end{proof}

\section{Groups without the triple vanishing property}
In this section, we construct pro-$p$-groups $G$ which do not have the vanishing triple Massey product property.
In particular, when $p=2$,  they are not realizable as $G_F(2)$ for any field $F$.

First we verify the following computational fact.

\begin{lem}
\label{lemma - triple commutator of matrices}
Let $a_i,b_i,c_i\in\F_p$, $i=1,2,3$, and set
\[
A=\begin{bmatrix}
1& 1 & a_1 & b_1\\
0& 1 & 0 & c_1\\
0& 0 & 1 & 0\\
0& 0 & 0 & 1
\end{bmatrix}, \quad
B=\begin{bmatrix}
1& 0 & a_2 & b_2\\
0& 1 & 1 & c_2\\
0& 0 & 1 & 0\\
0& 0 & 0 & 1
\end{bmatrix}, \quad
C=\begin{bmatrix}
1& 0 & a_3 & b_3\\
0& 1 & 0 & c_3\\
0& 0 & 1 & 1\\
0& 0 & 0 & 1
\end{bmatrix}.
\]
Then  $[[B,C],A]=\begin{bmatrix}
1& 0 & 0 & -1\\
0& 1 & 0 & 0\\
0& 0 & 1 & 0\\
0& 0 & 0 & 1
\end{bmatrix}$.
\end{lem}
\begin{proof}
A direct computation shows that $A^{-1},B^{-1},C^{-1}$ are
\[
\begin{bmatrix}
1& -1 & -a_1 & c_1-b_1\\
0& 1 & 0 & -c_1\\
0& 0 & 1 & 0\\
0& 0 & 0 & 1
\end{bmatrix}, \quad
\begin{bmatrix}
1& 0 & -a_2 & -b_2\\
0& 1 & -1 & -c_2\\
0& 0 & 1 & 0\\
0& 0 & 0 & 1
\end{bmatrix}, \quad
\begin{bmatrix}
1& 0 & -a_3 & a_3-b_3\\
0& 1 & 0 & -c_3\\
0& 0 & 1 & -1\\
0& 0 & 0 & 1
\end{bmatrix},
\]
respectively.
Therefore
\[
[B,C]= \begin{bmatrix}
1& 0 & 0 & a_2\\
0& 1 & 0 & 1\\
0& 0 & 1 & 0\\
0& 0 & 0 & 1
\end{bmatrix}, \quad
[B,C]^{-1}=\begin{bmatrix}
1& 0 & 0 & -a_2\\
0& 1 & 0 & -1\\
0& 0 & 1 & 0\\
0& 0 & 0 & 1
\end{bmatrix},
\]
and the assertion follows.
\end{proof}

\begin{ex}
\label{first example}
\rm
Let $S$ be a free pro-$p$-group on generators $x_1,\ldots,x_5$.
Let $r=[x_4,x_5][[x_2,x_3],x_1]$ and $\langle r\rangle$ the closed normal subgroup of $S$ generated by $r$.
Note that it is contained in the Frattini subgroup $S_{(2)}$ of $S$.
We show that $G=S/\langle r\rangle$ does not have the vanishing triple Massey product property.
To this end let
\[
A=\begin{bmatrix}
1& 1 & 0 & 0\\
0& 1 & 0 & 0\\
0& 0 & 1 & 0\\
0& 0 & 0 & 1
\end{bmatrix}, \quad
B=\begin{bmatrix}
1& 0 & 0 & 0\\
0& 1 & 1 & 0\\
0& 0 & 1 & 0\\
0& 0 & 0 & 1
\end{bmatrix}, \quad
C=\begin{bmatrix}
1& 0 & 0 & 0\\
0& 1 & 0 & 0\\
0& 0 & 1 & 1\\
0& 0 & 0 & 1
\end{bmatrix}.
\]
Let $\bar A,\bar B,\bar C$ be the images of $A,B,C$, respectively, in $\bar\U_4(\F_p)$.
We define a representation $\bar\rho\colon  S\to\bar\U_4(\F_p)$ by letting
\[
\bar\rho(x_1)=\bar A,\  \bar\rho(x_2)=\bar B,\  \bar\rho(x_3)=\bar C,\
 \bar\rho(x_4)=1,\ \bar\rho(x_5)=1.
\]
By Lemma \ref{lemma - triple commutator of matrices}, $\bar\rho(r)=[[\bar B,\bar C],\bar A]=1$,
so $\bar\rho$ induces a representation $\bar\rho\colon G\to\bar\U_4(\F_p)$.

Now suppose that $\rho\colon S\to\U_4(\F_p)$ is a representation such that $\rho_{i,i+1}=\bar\rho_{i,i+1}$ for $i=1,2,3$.
By Corollary \ref{cor:vanishing-lifting}, we need to show that $\rho(r)\neq1$.
We may write
\[
\rho(x_1)=\begin{bmatrix}
1& 1 & a_1 & b_1\\
0& 1 & 0 & c_1\\
0& 0 & 1 & 0\\
0& 0 & 0 & 1
\end{bmatrix},
\rho(x_2)=\begin{bmatrix}
1& 0 & a_2 & b_2\\
0& 1 & 1 & c_2\\
0& 0 & 1 & 0\\
0& 0 & 0 & 1
\end{bmatrix},
\rho(x_3)=\begin{bmatrix}
1& 0 & a_3 & b_3\\
0& 1 & 0 & c_3\\
0& 0 & 1 & 1\\
0& 0 & 0 & 1
\end{bmatrix},
\]
\[
\rho(x_4)=\begin{bmatrix}
1& 0 & a_4 & b_4\\
0& 1 & 0 & c_4\\
0& 0 & 1 & 0\\
0& 0 & 0 & 1
\end{bmatrix}, \quad
\rho(x_5)=\begin{bmatrix}
1& 0 & a_5 & b_5\\
0& 1 & 0 & c_5\\
0& 0 & 1 & 0\\
0& 0 & 0 & 1
\end{bmatrix}
\]
for some $a_i,b_i,c_i\in\F_p$, $i=1,2,3,4,5$.
We note that  $\rho(x_4)$ and $\rho(x_5)$ commute, so by Lemma \ref{lemma - triple commutator of matrices},
$\rho(r)=[[\rho(x_2),\rho(x_3)],\rho(x_1)]\neq1$, as claimed.
\qed
\end{ex}

\begin{ex}
\label{ex:ob}
Let $G$ be as in the previous example with $p=2$. Then by Theorem~\ref{thm:vanishing} (or more precisely, Theorem~\ref{thm:Massey vanishing}),  $G$ is not realizable as $G_F(2)$ for any field $F$.  For this statement,  using \cite{GLMS} we will give another proof, which avoids Theorem~\ref{thm:vanishing} and Massey product formalism technique. 

Assume that $G=G_F(2)$ for some field $F$. Note that $G$ is not a free prop-2 group, so  $F$ is of characteristic different from 2.  We denote $\sigma_i$ the image of $x_i$ in $G$. Let $\chi_1,\ldots,\chi_5\in H^1(S,\F_2)=H^1(G,\F_2)$ be the characters dual to $x_1,\ldots,x_5$. Let $[a_1],\ldots,[a_5]\in F^*/(F^*)^2$ be the corresponding elements to $\chi_1,\ldots,\chi_5$ via Kummer theory. This means $\chi_1(\sqrt{a_i})=\sqrt{a_i}$ for $i\not=1$ and $\chi_1(\sqrt{a_1})=-\sqrt{a_1}$, etc...

By \cite[Propositions 3.9.12-3.9.13]{NSW}, we have 
\[ a_2\cup a_3 = a_3\cup a_1=0.
\]
Consider a field $L/F$ attached to the triple $a_2,a_3,a_1$  (see \cite[Proposition 4.6]{GLMS}). Let $\sigma_{a_i}$ be constructed as in \cite[Proof of Proposition 4.6]{GLMS} with $a,b,c$ there replaced by $a_3,a_1,a_2$, respectively. Then $[[\sigma_{a_2},\sigma_{a_3}],\sigma_{a_1}]$ is a non-trivial element in $Z({\rm Gal}(L/F))\simeq \Z/2$. 
For each $i$, $\sigma_i$ and $\sigma_{a_i}$ act in the same way on $K=F(\sqrt{a_1},\sqrt{a_2},\sqrt{a_3})$. Therefore  $\sigma_i|_{L/F}=\sigma_{a_i}\gamma_i$, for $\gamma_i\in \Phi(\Gal(L/F))$ (here $\Phi(\Gal(L/F))$ is the Frattini subgroup of $\Gal(L/F)$).

In the proof of the claim below we use basic commutator identities together with basic identities valid in $\Gal(L/F)$.

{\bf Claim:} $[[\sigma_2,\sigma_3],\sigma_1]|_{L/F}=[[\sigma_{a_2},\sigma_{a_3}],\sigma_{a_1}]$.

In fact, 
\[
\begin{aligned}
{[}\sigma_2|_{L/F},\sigma_3|_{L/F}{]} &=[\sigma_{a_2}\gamma_2,\sigma_{a_3}\gamma_3]\\
&=[\sigma_{a_2},\sigma_{a_3}][\sigma_{a_2},\gamma_3][\gamma_2,\sigma_{a_3}]\\
&=[\sigma_{a_2},\sigma_{a_3}]c,
\end{aligned}
\]
where $c:=[\sigma_{a_2},\gamma_3][\gamma_2,\sigma_{a_3}]$ which is in the central of $\Gal(L/F)$. Hence
\[
\begin{aligned}
{[[}\sigma_2,\sigma_3],\sigma_1]|_{L/F}&= [[\sigma_2|_{L/F},\sigma_3|_{L/F}],\sigma_1|_{L/F}]\\
&= [[\sigma_{a_2},\sigma_{a_3}]c,\sigma_{a_1}\gamma_1]\\
&=[[\sigma_{a_2},\sigma_{a_3}],\sigma_{a_1}\gamma_1]\\
&= [[\sigma_{a_2},\sigma_{a_3}],\sigma_{a_1}].
\end{aligned}
\]
Therefore, $[[\sigma_2,\sigma_3],\sigma_1]|_{L/F}$ is a non-trivial element in $Z({\rm Gal}(L/F))\simeq \Z/2$. 

Also observe  that $[\sigma_4,\sigma_5]|_{L/F}$ is trivial because $\sigma_4,\sigma_5$ act trivially on $K=F(\sqrt{a},\sqrt{b},\sqrt{c})$ and the Galois group 
$\Gal(L/K)$ is abelian. Hence we see that our relation $[\sigma_4,\sigma_5][[\sigma_2,\sigma_3],\sigma_1]=1$ restricts nontrivially on $L/F$. Therefore we obtain a contradiction showing that $G\not\simeq G_F(2)$ for any field $F$.
\qed
\end{ex}

\begin{rmk}
 As noted in \cite{CEM, EM2},  one can use \cite[Proposition 9.1]{CEM} (or \cite[Corollary 6.3]{EM2})  to show that various pro-2-groups do not occur as $G_F(2)$ for some field $F$ of characteristic $\not=2$. For the convenience of the reader, we recall this result for pro-2-groups as below.
\begin{prop}[{\cite[Proposition 9.1]{CEM}, \cite[Corollary 6.3]{EM2}}] 
\label{prop:EM2}
Let $G_1, G_2$ be pro-2-groups such that $G_1/(G_1)_{(3)}\simeq G_2/(G_2)_{(3)}$ and $H^*(G_1,\F_2)\not\simeq H^*(G_2,\F_2)$. Then at most one of $G_1,G_2$ can be isomorphic to the maximal pro-2 Galois group $G_F(2)$ of a field $F$ of characteristic $\not=2$.
\end{prop}
To show that a pro-2-group $G_1$ cannot be isomorphic to $G_F(2)$ for a field $F$ of characteristic $\not=2$, we choose a group $G_2$ such that two conditions in the above Corollary are satisfied and $G_2$ does occur as $G_L(2)$ for some field $L$ of characteristic $\not=2$, and we are done.

Now we consider the pro-2-group $G=:G_1$ defined as in the previous example, i.e., $G$ is the quotient of the free pro-2 group $S$ on generators $x_1,\ldots,x_5$ by the relation $r=[x_4,x_5][[x_2,x_3],x_1]$. Then one might wonder whether we can use Proposition~\ref{prop:EM2} %\cite[Corollary 6.3]{EM2} 
to show that $G=G_1$ is not realizable as $G_F(2)$ for some field $F$ of characteristic $\not=2$. One very natural candidate for the group $G_2$ is the following: $G_2$ is the quotient of the free pro-2 group $S$ by the relation $r_2=[x_4,x_5]$. Then $G_1/(G_1)_{(3)}\simeq G_2/(G_2)_{(3)}$ and $G_2$ is the free product of the free pro-2 group on 3 generators $x_1,x_2,x_3$ with the group $\Z_2\times \Z_2$. And it is known that, see \cite[Theorem 3.6]{JW}, $G_2$ is isomorphic to $G_F(2)$ for some field $F$ of characteristic $\not=2$. However, $H^*(G_1,\F_2)\simeq H^*(G_2,\F_2)$. In fact, let 
\[
\begin{aligned}
H^1(G_1,\F_2)&=U_1\oplus V_1,\\
H^1(G_2,\F_2)&=U_2\oplus V_2,
\end{aligned}
\]
where for each $i=1,2$, $U_i$ is  spanned by the images of $\chi_1,\chi_2,\chi_3,\chi_4$ in $H^1(G_i,\F_2)$ and $V_i$ is spanned by the image of $\chi_5$ in $H^1(G_i,\F_2)$. Then using  usual transgression-relation pairing we see that:
\begin{enumerate}
\item The cup product $U_i\otimes U_i\to H^2(G_i,\F_2)$ is trivial.
\item The cup product $U_i\otimes V_i\to H^2(G_i,\F_2)$ is surjective. Because $\chi_4\cup\chi_5\not=0$ and $\dim H^2(G_i,\F_2)=1$.
\end{enumerate}
Hence $G_i$ are mild groups (see for example \cite{Fo,Ga,LM}). In particular, $\cd G_1=\cd G_2=2$ and $H^*(G_1,\F_2)=H^*(G_2,\F_2)$. Therefore we cannot easily apply Proposition
\ref{prop:EM2} to this example.

Our discussion above shows that our techniques provide genuinely new cases of pro-2-groups which cannot occur as $G_F(2)$ over some field $F$. 
Theorems~\ref{thm:ob1}  and~\ref{thm:ob2}  below exhibit  large families of pro-2-groups which are not of the form $G_F(2)$.

It is easy to provide examples as above with more relations. For example if $G=S/R$, where $S$ is the free pro-2-group on generators $x_1,x_2,\ldots,x_7$ and $R$ is its normal subgroup generated by 
$r_1=[x_4,x_5][[x_2,x_3],x_1] $
and $r_2=[x_6,x_7]$. Then the proof above for showing that $G\not\simeq G_F(2)$ for any field $F$ is valid word-for-word with the very exception that we choose in our possible example of $F$ an $\F_2$-basis $[a_1],\ldots,[a_7]$ orthogonal to $\chi_1,\ldots,\chi_7$ instead of the original basis $[a_1],\ldots,[a_5]$. 
\qed
\end{rmk}

Let $G$ be a pro-$p$-group. 
Let
\[ 1 \to R \to  S \to  G \to 1,\]
be a minimal presentation of $G$, i.e., $S$ a free pro-$p$-group and $R\subset S_{(2)}$. Then the inflation map
\[ \inf :H^1(G, \F_p) \to  H^1(S, \F_p)\]
is an isomorphism by which we identify both groups. Since $S$ is free, we have $H^2(S,\F_p) = 0$ and from the  5-term exact sequence we obtain the transgression map 
\[ \trg : H^1(R,\F_p)^ G \to H^2(G,\F_p)\]
is an isomorphism. 
Therefore any element $r\in R$ gives rise to a map
\[ \tr_r: H^2(G,\F_p) \to \F_p,\]
which is defined by $\alpha\mapsto \trg^{-1}(\alpha)(r)$ and is called the {\it trace map} with respect to $r$.

Let $(x_i)_{i\in I}$ be a basis of $S$, where $I$ is a well-ordered set. Let $\chi_i,i \in I$ be the dual basis to $x_i, i\in I$ of $H^1(S,\F_p)=H^1(G,\F_p)$, i.e., $\chi_i(x_j)=\delta_{ij}$.

Let $r$ be any element in $S_{(2)}$. Then $r$  may be uniquely written as
\begin{equation}
\label{decomposition modulo S4}
\tag{*}
r=
\begin{cases}
\displaystyle \prod_{i\in I} x_i^{2a_i} \prod_{ i<j}[x_i,x_j]^{b_{ij}}\prod_{ i< j, k\leq j}[[x_i,x_j],x_k]]^{c_{ijk}}\cdot r', \text{ if } p=2,\\
\displaystyle  \prod_{i<j}[x_i,x_j]^{b_{ij}}\prod_{i\in I} x_i^{3a_i}  \prod_{i< j, k\leq j}[[x_i,x_j],x_k]]^{c_{ijk}}\cdot r', \text{ if } p=3,\\
\displaystyle  \prod_{i<j}[x_i,x_j]^{b_{ij}}\prod_{ i< j, k\leq j}[[x_i,x_j],x_k]]^{c_{ijk}}\cdot r', \text { if } p\not=2,3,
\end{cases}
 \end{equation}
where $a_i,b_{ij},c_{ijk}\in \{0,1,\ldots,p-1\}$ and $r'\in S_{(4)}$ \cite[Prop.\ 1.3.2 and Prop.\ 1.3.3]{Vo}. 
For convenience we call (\ref{decomposition modulo S4}) the canonical decomposition modulo $S_{(4)}$ of $r$ (with respect to the basis $(x_i)$) and we also set $u_{ij}=b_{ij}$ if $i<j$, and  $u_{ij}=b_{ji}$ if $j<i$.

\begin{lem} 
\label{lem:coef}
Let the notation  be as above. Assume that $R=\langle r \rangle$ and that the triple Massey product $\langle -\chi_k,-\chi_i,-\chi_j\rangle$ is defined for some distinct $i,j,k$ with $i<j$ and $k<j$. Then there exists an $\alpha\in \langle -\chi_k,-\chi_i,-\chi_j\rangle$, which can be given explicitly,  such that
\[
\tr_r(\alpha)=c_{ijk}.
\]
\end{lem}
\begin{proof} Since $\langle -\chi_k,-\chi_i,-\chi_j\rangle$ is defined, $\chi_k\cup \chi_i=\chi_i\cup\chi_j=0$. Hence by \cite[Proposition 1.3.2]{Vo} (see also \cite[Proposition 3.9.13]{NSW}), we have $u_{ki}=u_{ij}=0$. 

Let
\[
A=\begin{bmatrix}
1& 1 & 0 & 0\\
0& 1 & 0 & 0\\
0& 0 & 1 & 0\\
0& 0 & 0 & 1
\end{bmatrix}, \quad
B=\begin{bmatrix}
1& 0 & 0 & 0\\
0& 1 & 1 & 0\\
0& 0 & 1 & 0\\
0& 0 & 0 & 1
\end{bmatrix}, \quad
C=\begin{bmatrix}
1& 0 & 0 & 0\\
0& 1 & 0 & 0\\
0& 0 & 1 & 1\\
0& 0 & 0 & 1
\end{bmatrix}.
\]
Then $A^p=B^p=C^p=[A,C]=1$.

We define a representation $\rho \colon S\to \U_4(\F_p)$ by letting
\[
x_k\mapsto  A, \ x_i\mapsto   B, \ x_j\mapsto  C,\
 x_l\mapsto 1,\quad \forall l\not=i,j,k.
\]
Then 
\[
\rho(r) = [A,C]^{u_{kj}}[[B,C],A]^{c_{ijk}}\cdot [[A,C],B]^{c_{kji}}= [[B,C],A]^{c_{ijk}}.
\]
Hence $\rho(r)=1$ in $\bar\U_4(\F_p)$. Thus  $\rho$ induces a group homomorphism $\bar\rho: G\to \bar \U_4(\F_p)$. Then $\rho$ is a lift of $\bar\rho$ in the sense discussed before Lemma~\ref{lem:Sharifi}. By checking on the generators we see that 
\[
\bar\rho_{12}=\chi_k, \bar\rho_{23}=\chi_i,\bar\rho_{34}=\chi_j.
\]

Let $\alpha\in \langle -\chi_k,-\chi_i,-\chi_j\rangle$ be the Massey product value relative to the defining system corresponding to $\bar\rho$ and let $f\in H^1(R,\F_p)^G$ be defined by 
\[
 f(\tau)= -\rho_{14}(\tau) \quad \text{ for } \tau\in R.
\]
By Lemma~\ref{lem:Sharifi}, we have $\trg(f)=\alpha$. Hence 
\[
\tr_r(\alpha)= f(r)=-\rho_{14}(r)=c_{ijk},
\]
as desired.
\end{proof}

\begin{prop}
\label{prop:coef1}
In (\ref{decomposition modulo S4}) suppose that there exist distinct $i,j,k$ such that $i<j,k<j$,
$u_{ij}=u_{kj}=u_{ki}=u_{kl}=u_{jl}=0$  for all $l\neq i,j,k$. If $p=2$ we assume further that $a_k=a_j=0$.
Let $G=S/\langle r\rangle$ and $\chi_1,\ldots,\chi_n$ be the $\F_p$-basis of $H^1(S,\F_p)= H^1(G,\F_p)$ dual to $x_1,\ldots, x_n$. Then
$\langle -\chi_k,-\chi_i,-\chi_j\rangle$ is uniquely defined and 
\[ \tr_r(\langle -\chi_k,-\chi_i,-\chi_j\rangle)= c_{ijk}.
\]
In particular, if we assume further that $c_{ijk}\not=0$ then $\langle -\chi_k,-\chi_i,-\chi_j\rangle$ does not vanish.
\end{prop}
\begin{proof}
 By \cite[Proposition 1.3.2]{Vo} (see also \cite[Proposition 3.9.13]{NSW}) and by assumption, we have
 \[ 
{\rm tr}_r( \chi_k\cup \chi_l)= 
 \begin{cases}
 \pm u_{kl} =0 &\text{ if } l\not =k,\\
 0 &\text{ if } l =k, p\not =2,\\
a_k =0 &\text{ if } l=k, p=2.%\
 \end{cases}
 \]

Hence $\chi_k\cup \chi_l=0$ for all $l=1,\ldots,n$. Thus $\chi_k\cup H^1(G,\F_p)=0$. Similarly $H^1(G,\F_p)\cup \chi_j=0$. Therefore $\langle -\chi_k,-\chi_i,-\chi_j\rangle$ is uniquely defined. By Lemma~\ref{lem:coef}, 
\[
\tr_r(\langle -\chi_k,-\chi_i,-\chi_j\rangle) = c_{ijk},
\]
as desired.
\end{proof}

The following theorem generalizes Example \ref{first example}.

\begin{thm}
\label{thm:ob1}
Let $\sR$ be a set of elements in $S_{(2)}$. Assume that there exists an element $r$ in $\sR$ and distinct indices $i,j,k$ with $i<j,k<j$ such that: 
\begin{enumerate}
\item  In (\ref{decomposition modulo S4})  the canonical decomposition modulo $S_{(4)}$ of $r$,  
$u_{ij}=u_{kj}=u_{ki}=u_{kl}=u_{jl}=0$  for all $l\neq i,j,k$, and  $c_{ijk}\not=0$ and if $p=2$ we assume further that $a_k=a_j=0$, and
\item  for every $s\in \sR$ which is different from  $r$,  the factors $[x_k,x_i]$, $[x_i,x_k]$ and $[x_i,x_j]$ do not occur in the canonical decomposition modulo $S_{(4)}$ of $s$. 
\end{enumerate}
Then $G=S/\langle \sR \rangle$ does not have the vanishing triple Massey product property.
\end{thm}
\begin{proof} Let $G^\prime=G/\langle r \rangle$ and let $f$ be the canonical map $f:G^\prime=S/\langle r\rangle \to G= S/\langle \sR \rangle$. We shall identify three groups $H^1(S,\F_p)$, $H^1(G,\F_p)$ and $H^1(G^\prime,\F_p)$ via  inflation maps. We also use subscript $\langle \cdot,\cdot,\cdot\rangle_G $ (respectively, $\langle \cdot,\cdot,\cdot\rangle_{G^\prime}$) to denote Massey products in the cohomology groups of $G$ (respectively,  $G^\prime$).

By \cite[Proposition 1.3.2]{Vo} (see also \cite[Proposition 3.9.13]{NSW}) and by assumption, we have 
\[
\tr_{s}(\chi_k\cup \chi_i)=\tr_{s}(\chi_i\cup\chi_j)=0, \;\text{ for all } s\in \sR.
\]
Hence $\chi_k\cup \chi_i=\chi_i\cup\chi_j=0$ and the triple Massey products $\langle -\chi_k,-\chi_i,-\chi_j\rangle_G$ is defined. 

By the naturality property of Massey products (see e.g. \cite[page 433]{Kra}, \cite[Property 2.1.2]{Mor}), one has 
\[ 
f^*(\langle -\chi_k,-\chi_i,-\chi_j\rangle_{G}) \subseteq  \langle -\chi_k,-\chi_i,-\chi_j\rangle_{G^\prime}. 
\]
 By Proposition~ \ref{prop:coef1} applying to the group $G^\prime$,  $\langle -\chi_k,-\chi_i,-\chi_j\rangle_{G^\prime}$ does not vanish. Therefore  $\langle -\chi_k,-\chi_i,-\chi_j\rangle_{G}$ does not vanish, and we are done. 
\end{proof}
\begin{lem}
\label{lem:ob2}
Let 
$u:=\begin{bmatrix}
1& 0 & 0 & 0\\
0& 1 & 1 & 0\\
0& 0 & 1 & 1\\
0& 0 & 0 & 1
\end{bmatrix}, \; 
 v:=\begin{bmatrix}
1& 1 & 0 & 0\\
0& 1 & 0 & 0\\
0& 0 & 1 & 0\\
0& 0 & 0 & 1
\end{bmatrix},
$
be two matrices in $\U_4(\F_p)$.
Then $[[u,v],u]= \begin{bmatrix}
1& 0 & 0 & -1\\
0& 1 & 0 & 0\\
0& 0 & 1 & 0\\
0& 0 & 0 & 1
\end{bmatrix}$,$[[u,v],v]=1$ and  $v^p=1$. Furthermore, if $p\geq 3$ then $u^p=1$.
\end{lem}
\begin{lem} 
\label{lem:coef2}
Let the notation  be as in Lemma~\ref{lem:coef}. Assume that $R=\langle r\rangle$. 
\begin{enumerate}
\item  Assume that the triple Massey product $\langle -\chi_j,-\chi_i,-\chi_i\rangle$ is defined for some  $i<j$. Then there exists an $\alpha\in \langle -\chi_j,-\chi_i,-\chi_i\rangle$ such that  $tr_r(\alpha)=c_{iji}.$
\item Assume that the triple Massey product $\langle -\chi_i,-\chi_j,-\chi_j\rangle$ is defined for some $i<j$ . Then there exists an $\alpha\in \langle -\chi_i,-\chi_j,-\chi_j\rangle$ such that 
$tr_r(\alpha)=c_{ijj}.$
\end{enumerate}
\end{lem}
\begin{proof} 
We only prove (1) since (2) can be proved similarly. Since $\langle -\chi_j,-\chi_i,-\chi_i\rangle$ is defined, $\chi_i\cup\chi_j=0=\chi_i\cup\chi_j$. Hence by \cite[Proposition 1.3.2]{Vo} (see also \cite[Proposition 3.9.13]{NSW}), $u_{ij}=0$  and if $p=2$ then $a_i=0$. 

Let $u,v$ be two matrices as in Lemma~\ref{lem:ob2}.  
We define a representation $\bar{\rho}: S\to\bar{\U}_4(\F_p)$ by letting
\[
x_i\mapsto u, x_j \mapsto \bar v,
 x_l\mapsto 1,\quad \forall l\not=i,j. 
\]

Then 
\[
\rho(r)= [[u,v],u]^{c_{iji}} [[u,v],v]^{c_{ijj}}=[[u,v],u]^{c_{iji}}.
\]
Hence $\rho(r)=1$ in $\bar\U_4(\F_p)$. Thus  $\rho$ induces a group homomorphism $\bar\rho: G\to \bar \U_4(\F_p)$. Then $\rho$ is a lift of $\bar\rho$ in the sense discussed before Lemma~\ref{lem:Sharifi}. By checking on the  generators we see that 
\[
\bar\rho_{12}=\chi_j,\quad \bar\rho_{23}=\chi_i,\quad \bar\rho_{34}=\chi_i.
\]

Let $\alpha\in \langle -\chi_j,-\chi_i,-\chi_i\rangle$ be the Massey product value relative to the defining system corresponding to $\bar\rho$ and let $f\in H^1(R,\F_p)^G$ be defined by 
\[
 f(\tau)= -\rho_{14}(\tau) \quad \text{ for } \tau\in R.
\]
By Lemma~\ref{lem:Sharifi}, we have $\trg(f)=\alpha$. Hence 
\[
\tr_r(\alpha)= f(r)=-\rho_{14}(r)=c_{iji},
\]
as desired.
\end{proof}

\begin{prop}
\label{prop:coef2}

In (\ref{decomposition modulo S4}) suppose that there exist $i<j$ such that  $u_{ij}=0=u_{il}=u_{jl}$, for all $l\not=i,j$. If $p=2$ we assume further that $a_i=a_j=0$.
Let $G=S/\langle r\rangle$ and $\chi_1,\ldots,\chi_n$ be the $\F_p$-basis of $H^1(S,\F_p)= H^1(G,\F_p)$ dual to $x_1,\ldots, x_n$. Then $\langle -\chi_j,-\chi_i,-\chi_i\rangle$ and $\langle -\chi_i,-\chi_j,-\chi_j\rangle$ are uniquely defined and we have
\[ \tr_r(\langle -\chi_j,-\chi_i,-\chi_i\rangle)= c_{iji}.
\]
and 
\[ \tr_r(\langle -\chi_i,-\chi_j,-\chi_j\rangle)= c_{ijj}.
\]
In particular, if we assume further that $c_{iji}\not=0$ (respectively, $c_{ijj}\not=0$) then $\langle -\chi_j,-\chi_i,-\chi_i\rangle$ (respectively, $\langle -\chi_i,-\chi_j,-\chi_j\rangle$) does not vanish.
\end{prop}

\begin{proof} Under our assumption the triple Massey products $\langle -\chi_j,-\chi_i,-\chi_i\rangle $ and $\langle -\chi_i,-\chi_j,-\chi_j\rangle$ are uniquely defined. Then by Lemma~\ref{lem:coef2}, 
\[ \tr_r(\langle -\chi_j,-\chi_i,-\chi_i\rangle)= c_{iji}.
\]
and 
\[ \tr_r(\langle -\chi_i,-\chi_j,-\chi_j\rangle)= c_{ijj},
\]
as desired.
\end{proof}

\begin{thm}
\label{thm:ob2}
Let $\sR$ be a set of elements in $S_{(2)}$. Assume that there exists an element $r$ in $\sR$ and distinct indices $i,j$ with $i<j$ such that: 
\begin{enumerate}
\item  In (\ref{decomposition modulo S4}),
$u_{ij}=u_{il}=u_{jl}=0$, for all $l\not=i,j$ and $c_{iji}\not=0$ (respectively, $c_{ijj}\not=0$) and if $p=2$ we assume further that $a_i=a_j=0$, and
\item  for every $s\in \sR$ which is different from  $r$,  the factor $[x_i,x_j]$  does not occur in the canonical decomposition modulo $S_{(4)}$ of $s$ and if $p=2$ we further assume that $x_i^2$ (respectively, $x_j^2$) does not occur  in the canonical decomposition modulo $S_{(4)}$ of $s$.
\end{enumerate}
Then $G=S/\langle \sR \rangle$ does not have the vanishing triple Massey product property.
\end{thm}
\begin{proof} 
Let $G^\prime=G/\langle r \rangle$ and let $f$ be the canonical map $f:G^\prime=S/\langle r\rangle \to G= S/\langle \sR \rangle$. We shall identify three groups $H^1(S,\F_p)$, $H^1(G,\F_p)$ and $H^1(G^\prime,\F_p)$ via  inflation maps. We also use subscript $\langle \cdot,\cdot,\cdot\rangle_G $ (respectively, $\langle \cdot,\cdot,\cdot\rangle_{G^\prime}$) to denote Massey products in the cohomology groups of $G$ (respectively,  $G^\prime$).

We only treat the case that $c_{iji}\not=0$. The other case is treated similarly. 

By \cite[Proposition 1.3.2]{Vo} (see also \cite[Proposition 3.9.13]{NSW}) and by assumption, we have 
\[
\tr_{s}(\chi_j\cup \chi_i)=\tr_{s}(\chi_i\cup\chi_i)=0, \;\text{ for all } s\in \sR.
\]
Hence $\chi_j\cup \chi_i=\chi_i\cup\chi_i=0$ and the triple Massey product $\langle -\chi_j,-\chi_i,-\chi_i\rangle_G$ is defined. 

By the naturality property of Massey products (see e.g., \cite[page 433]{Kra}, \cite[Property 2.1.2]{Mor}), one has 
\[ 
f^*(\langle -\chi_j,-\chi_i,-\chi_i\rangle_{G}) \subseteq  \langle -\chi_j,-\chi_i,-\chi_i\rangle_{G^\prime}. 
\]
 By Proposition~ \ref{prop:coef2} applying to the group $G^\prime$,  $\langle -\chi_j,-\chi_i,-\chi_i\rangle_{G^\prime}$ does not vanish. Therefore  $\langle -\chi_j,-\chi_i,-\chi_i\rangle_{G}$ does not vanish, and we are done. 
\end{proof}

\section{Further directions}
\label{sec:further}
Let $p$ be a prime number. Let $F$ be a field of characteristic $\not=p$, which contains a primitive $p$-th root of unity. Let $G=G_F(p)$ be the maximal pro-$p$-quotient of  the absolute Galois group $G_F$ of $F$. Denote by $G_{(i)},i=1,2,\ldots$  the $p$-Zassenhaus filtration of $G$. Let $F_{(i)}$ be the fixed field $F(p)^{G_{(i)}}$ of the group $G_{(i)}$, where $F(p)$ is the maximal  $p$-extension of $F$. 

When $p=2$, $F_{(3)}$ is the compositum of all $C_2,C_4,D_4$-extensions $K/F$ inside $F(2)$. This fact was proved by Villegas \cite{Vi} and \cite[Corollary 2.18]{MS2} (see also \cite[Corollary 11.3]{EM1} for a more general result). Inspired by this beautiful fact, and the second proof of Theorem~\ref{thm:vanishing}, we would like to propose the following conjecture.

Let $C_n$ be the cyclic group of order $n$, $D_4$ the dihedral group of order $8$ and let $G_1$ and $G_2$ be  groups defined as in \cite{GLMS} (see Cases 3 and 5 of the second proof  of Theorem~\ref{thm:vanishing} for the definition). Explicitly, $G_2\simeq \U_4(\F_2)$ and $G_1\simeq$ the subgroup of $\U_4(\F_2)$ consisting of upper $4\times 4$-matrices $(a_{ij})$ with $a_{23}=a_{34}$.

\begin{conj}
 Let the notation be as above with $p=2$. Then $F_{(4)}$ is the  compositum of $C_2,C_4$, $D_4,G_1,G_2$-extensions $K/F$ inside $F(2)$. 
 \end{conj}

We define the field $F_{\omega}$ as the compositum of $C_2,C_4, D_4,G_1,G_2$-extensions $K/F$ inside $F(2)$. Then $F_{\omega}\subset F_{(4)}$ and the conjecture says  that in fact $F_{\omega}=F_{(4)}$.  

\begin{defn}
 Let $G$ be a pro-$p$-group and let $n\geq 1$ be an integer. We say that $G$ has the {\it kernel $n$-unipotent property} if 
\[
G_{(n)}=\bigcap \ker(\rho\colon G\to \U_n(\F_p)),
\]
where $\rho$ runs over the set of all representations (continuous homomorphisms) $G\to \U_n(\F_p)$.
\end{defn}

It is easy to see that for $n=1,2$, every pro-$p$-group $G$ has the kernel $n$-unipotent property. It was shown that for $G=G_F(p)$, where $F$ is a field containing a primitive $p$-root of unity, $G$ has the kernel $3$-unipotent property. (See \cite{MS2,Vi, EM1} for the case $p=2$ and \cite[Example 9.5(1)]{EM2} for the case $p>2$.)
For any fixed integer $n\geq 3$, in \cite{MT2} we also give an example of a torsion free pro-$p$-group $G$ such that $G$ does not have the kernel $n$-unipotent property.

The following conjecture is a generalization of the above conjecture.
\begin{conj}[Kernel $n$-Unipotent Conjecture]
\label{conj:kernel intersection}
Let $F$ be a field containing  a primitive $p$-th root of unity and let $G=G_F(p)$. Let $n\geq 3$ be an integer. Then $G$ has the kernel $n$-unipotent property.
\end{conj}

In a subsequent paper \cite{MT}, we show that every pro-$p$ Demushkin group has the kernel $4$-unipotent property. In \cite{MT2}, we also show that pro-$p$ Demushkin groups of rank 2 have the kernel $n$-unipotent property for all $n\geq 4$.  It is shown in \cite[Theorem A]{Ef} that every free pro-$p$-group has the kernel $n$-property for all $n\geq 3$. 
(In \cite{MT2} we provide an alternative direct short proof.) 

The results of this paper are also relevant in determining strong automatic realizations of  canonical quotients of absolute Galois groups. (See \cite{MST}.)

Finally it is very interesting to extend the main theorems in this paper also to the case $p > 2$. (See \cite{GMT}.)

\end{document}